\documentclass{amsart}
\usepackage{amssymb}
\usepackage{adjustbox}
\calclayout
\usepackage{etoolbox}
\usepackage{mathtools}
\usepackage{epsfig}  		
\usepackage{epic,eepic}       
\usepackage{tikz}
\usepackage{stmaryrd}
\usepackage{url}

\newtheorem{theorem}{Theorem}[section]
\newtheorem{proposition}[theorem]{Proposition}
\newtheorem{lemma}[theorem]{Lemma}
\newtheorem{cor}[theorem]{Corollary}

\theoremstyle{definition}

\theoremstyle{remark}
\newtheorem{remark}[theorem]{Remark}

\numberwithin{equation}{section}

\makeatletter
\let\origmaketitle\maketitle
\renewcommand{\maketitle}{\vspace*{-45pt}\origmaketitle}
\makeatother
\begin{document}

\title[Random multiplicative functions in the supercritical regime]{Partial sums of random multiplicative functions \\with supercritical divisor twists}

\author{Jad Hamdan}
\address{Mathematical Institute, University of Oxford}
\email{hamdan@maths.ox.ac.uk}

\begin{abstract}
Let $f$ be a Steinhaus random multiplicative function, and for $\alpha\in \mathbb{R}$, let $d_\alpha$ denote the $\alpha$-divisor function. For $\alpha\in(1,2)$ we establish that
\[
\mathbb{E}\bigg\{\Big|\frac{1}{\sqrt{x}}\sum_{n\leq x} d_\alpha(n)f(n)\Big|^{2q}\bigg\}
\ll
\frac{(\log x)^{2q(\alpha-1)}}{(\log\log x)^{3\alpha q/2}(1-\alpha q)+1}
\]
uniformly for $q\in [0,1/\alpha]$ and all large $x$. This matches predictions from the theory of {supercritical} Gaussian multiplicative chaos, and provides an analogue of a seminal result of Harper corresponding to the {critical} ($\alpha=1$) case. 

Our approach is based on studying the measure of level sets of an Euler product associated with $f$, and yields a short proof of Harper’s upper bound at $\alpha=1$ (implying Helson's conjecture at $q=1/2$). As an additional application, we obtain a conjecturally sharp bound for the pseudomoments of the Riemann zeta function in a certain parameter range, showing that 
\[
\lim_{T\to\infty}\frac{1}{T}\int_T^{2T}
\bigg|\sum_{n\leq x}\frac{d_\alpha(n)}{n^{1/2+it}}\bigg|^{2q} \mathrm{d}t
\ll
\frac{(\log x)^{2q(\alpha-1)}}{(\log\log x)^{3\alpha q/2}},
\]
for $\alpha\in (1,2)$ and small $q>0$.
This answers a question of Gerspach.
\end{abstract}

\maketitle

\vspace{-20px}

\section{Introduction}

Let $(Z_p)_p$ denote a sequence of independent and identically distributed random variables indexed by the primes, which are uniformly distributed on the complex unit circle. A \textit{Steinhaus random multiplicative function} $f:\mathbb{N}\to \mathbb{C}$ is defined as $f(p)=Z_p$ on the primes, and extended to $\mathbb{N}$ by making $f$ completely multiplicative. Originally introduced to model Archimedean characters $n\mapsto n^{it}$, the study of partial sums of $f$ and other, similarly defined random functions has grown into an active area of research in its own right \cite{OfirMoDick1,OfirMoDick2,OfirMoDickConvergence,MR4682955,HarperRMF,HarperII}. 

A landmark result in this area is Harper's \cite{HarperRMF} resolution of Helson's conjecture \cite{Helson}, which states that partial sums of $f(n)$ display a surprising amount of cancellation when compared to, say, sums of independent random variables. To be precise, Helson conjectured that $\mathbb{E}|\sum_{n\leq x}f(n)|=o(\sqrt{x})$, and Harper showed that the following holds uniformly in $q\in [0,1]$:
\begin{equation}\label{eq:Harper}
        \mathbb{E}\bigg\{\Big|\frac{1}{\sqrt{x}}\sum_{n\leq x}f(n)\Big|^{2q}\bigg\}\asymp \big(\sqrt{\log\log x}(1-q)+1\big)^{-q}.
\end{equation}
The key insight in \cite{HarperRMF} is that this phenomenon can be explained using a connection to the theory of  Gaussian multiplicative chaos (GMC), whose relevance to number theory first emerged with the conjectures of Fyodorov, Hiary and Keating on the magnitude of the Riemann zeta function on typical short intervals of the critical line \cite{FyodorovHiaryKeating, FyodorovKeating}. 

Harper's argument is comprised of two main steps. The first consists of comparing moments of $x^{-1/2}\sum_{n\leq x}f(n)$ to those of the integral of a random Euler product, by a sieve-theoretic argument and Parseval's theorem. The second is to realise that this integral approximates the total mass of a random measure---a smooth perturbation of  \textit{critical} GMC---only when scaled by $\sqrt{\log\log x}$, which is known as the Seneta-Heyde normalisation of critical GMC \cite{PowellSurvey}. The absence of this factor delivers the sought-after cancellation, and explains the right-hand side in \eqref{eq:Harper}. The bulk of the work lies in making this connection rigorous, which the author achieves by proving a non--Gaussian analogue of Girsanov's theorem, setting the stage for an application of the classical {ballot theorem}. A recent work of Gorodetsky and Wong \cite{GWHelson} showed that by appealing to pre-existing results in the GMC literature instead, namely Kahane's convexity inequality and a coupling due to Saksman and Webb \cite{SaksmanWebb}, one can significantly shorten this Gaussian comparison step and recover \eqref{eq:Harper}, albeit without the uniformity in $q$ near $1$.

The purpose of the current work is to present a new and self-contained way to establish this comparison, through the study of \textit{large deviations} of random Euler products. We use this to give a new and short proof of \eqref{eq:Harper} and, more importantly, to prove a conjecturally sharp upper bound for partial sums of $f\cdot d_\alpha$, where $d_\alpha$ is the $\alpha$-divisor function and $\alpha\in (1,2)$. In this regime, which corresponds to the \textit{supercritical} phase of GMC, recovering double-logarithmic corrections similar to \eqref{eq:Harper} requires a precise understanding of the maximum of random Euler products in short intervals. This also has applications to the study of \textit{pseudomoments} of the Riemann zeta function, discussed further below.

\subsection{Main results} Let $f$ denote a Steinhaus random multiplicative function, and $d_\alpha$ denote the $\alpha$-divisor function, defined through $\zeta(s)^\alpha=\sum_{n\geq 1}d_\alpha(n)n^{-s}$ where this series converges. Our main result is the following supercritical analogue of \eqref{eq:Harper}.

\begin{theorem}\label{thm:main} Fix $\alpha\in (1,2)$. Then uniformly in all large $x$ and $q\in [0,1/\alpha]$,
\[
    \mathbb{E}\bigg\{\Big|\frac{1}{\sqrt{x}}\sum_{n\leq x} d_\alpha(n)f(n)\Big|^{2q}\bigg\}\ll \frac{(\log x)^{2q(\alpha-1)}}{(\log\log x)^{(3\alpha q/2)}(1-\alpha q)+1}.
\]
\end{theorem}

\noindent The proof proceeds by studying averages of $|F_x|^{2\alpha}$, where $F_x(s):=\prod_{p\leq x}(1-f(p)p^{-s})^{-1}$ is the Euler product associated to $f$, truncated at $x$. Harper's argument in \cite{HarperRMF} carries this out when $\alpha=1$, comparing the left-hand side to the $q$-th moment of $\int_0^1|F_x(\tfrac{1}{2}+ih)|^2\mathrm{d}h$. For larger $\alpha$, the same comparison naturally gives rise to moments of $\int_0^1|F_x(1/2+ih)|^{2\alpha}\mathrm{d}h$, and more generally suggests that the order of magnitude of partial sums twisted by a multiplicative function 
$g$ is governed by averages of $|F_x|^{\gamma}$  whenever
\begin{equation}\label{eq:growth}
        \sum_{p\leq t} |g(p)|^2 \sim \frac{\gamma}{2}\frac{t}{\log t}
\end{equation}
with a sufficiently strong error term (see \cite{Gerspach}, and the introduction of \cite{OfirMoDick1}). The divisor function $d_\alpha$
  provides the simplest example of such a twist (for which $\gamma=2\alpha$).

To illustrate the method, we begin by proving the following Euler product bound at $\gamma=2$.
\begin{proposition}\label{prop:2bound} Uniformly in $x$ sufficiently large and $q\in [0,1]$,
\begin{equation}\label{eq:2bound}
    \mathbb{E}\bigg\{\bigg(\frac{1}{\log x}\int_{0}^1 \big|F_x(1/2+ih)\big|^2\mathrm{d}h\bigg)^q\bigg\}\ll \big({\sqrt{\log\log x}(1-q)+1}\big)^{-q}.
\end{equation}
\end{proposition}
\noindent Our approach will consist of studying this integral through the measure of level sets of $\log |F_x(s)|$. This will allow us to bypass the need for an approximate Girsanov's theorem as in \cite{HarperRMF}, and reveals that the integral is dominated by those $h$ for which $$\log|F_x(1/2+ih)|=\log \log x-O(\sqrt{\log\log x}).$$ Combining this with the first step of Harper's argument in \cite{HarperRMF} yields a short proof of the upper bound therein (and thus of Helson's conjecture), which is given in full in Section \ref{sec:Helson}.

We then adapt this approach to integrals of $|F_x|^{\gamma}$ for $\gamma>2$, where the analysis becomes more delicate. We also handle averages off the critical line.
 
\begin{theorem}\label{thm:mainbound} Fix $\gamma\in (2,4)$. Then uniformly in $q\in [0,2/\gamma]$, all large $x$, $0\leq k\leq \lfloor \log\log\log x\rfloor{+1}$, $\sigma_k=1/2-k/\log x$, and $3<y\leq x^{e^{-k}}$,
\begin{equation}\label{eq:mainbound}
    \mathbb{E}\bigg\{\bigg(\frac{1}{\log y}\int_{0}^1 \big|F_y(\sigma_k+ih)\big|^\gamma\mathrm{d}h\bigg)^q\bigg\}\ll \frac{(\log y)^{(\gamma-2)q}}{(\log\log y)^{(3\gamma q/4)}(2-\gamma q)+1}.
\end{equation}
\end{theorem}
Upon making the necessary identifications, the exponents on the right-hand side match those found in the normalisation of \textit{supercritical} GMC \cite{glassy}, which is also known to only have moments up to $q<2/\gamma$.  By contrast with the critical ($\gamma=2$) case, the dominant contribution to the integral will now come from points surrounding the local maxima of $|F_y(\sigma_k+ih)|$ on $h\in [0,1]$. This leads us to study $\max_{h\in [0,1]}|F_y(\sigma_k+ih)|$ using ideas from the study of extrema of log-correlated fields \cite{BramsonDingZeitouni,Bramson}, and the following bound arises as an immediate corollary of the analysis.

\begin{cor}[Maximum bound]\label{cor:maximum} For $x$ large and $1<y\leq \log\log x/\log\log\log x$,
\[
    \mathbb{P}\bigg(\max_{h\in [0,1]}|F_x(1/2+ih)|>\frac{\log x}{(\log\log x)^{3/4}}e^y\bigg)\ll y\exp{\Big(\!-2y-\frac{y^2}{\log\log x}\Big)}.
\]
\end{cor}

\noindent This can be seen as a random analogue of the Fyodorov-Hiary-Keating conjecture \cite{FyodorovHiaryKeating,FyodorovKeating}, matching the best known bound in that setting \cite{FHK1}, and improving the one for $\max_{h\in [0,1]}|F_x(1/2+ih)|$ in \cite{ABH} to $O(1)$ precision. By understanding the structure of these maxima, we can also bound the typical measure of level sets of $\log |F_x|$ near the height of the maxima.
\begin{cor}[Typical measure of level sets]\label{cor:level} Let $x$ be large and {$1<A\leq (\log x)^{1/\log\log\log x}$}. Then uniformly in $|y|<(\log\log x)/2$,
\[
\mathrm{meas}\Big\{h\in [0,1]: |F_x(1/2+ih)|>\frac{\log x}{(\log\log x)^{3/4}}e^y\Big\}\leq (\log x)^{-1}A|\log A-y|\exp{\Big(\!-2y-\frac{y^2}{\log\log x}\Big)}
\]
with probability $1-O((\log A)/A)$.
\end{cor}
\noindent Theorem \ref{thm:mainbound} and both of these corollaries display what is typical of Gaussian log-correlated processes (see, e.g.,\ Lemma 4.2 in \cite{CHL}), and are expected to be sharp. In particular, by taking $A$ sufficiently large and $y=O(1)$ in Corollary \ref{cor:level}, we find that the measure of points $h\in [0,1]$ for which $\log|F_x(1/2+ih)|=\log\log x-\frac{3}{4}\log\log\log x+O(1)$ is  $\ll_A(\log x)^{-1}$ with high probability. This suggests that their contribution to the left-hand side in \eqref{eq:mainbound} should match the upper bound.



\bigskip
Lastly, the $(\log\log x)^{3\gamma q/4}$ saving in Theorem \ref{thm:mainbound} leads to improved bounds for the \textit{pseudomoments} 
\[
\Psi_{2q,\alpha}(x):= \lim_{T\to\infty} \frac{1}{T}\int_{T}^{2T} \bigg|\sum_{n\leq x}\frac{d_\alpha(n)}{n^{1/2+it}}\bigg|^{2q}\mathrm{d}t
\]
of the Riemann zeta function, which were first introduced by Conrey and Gamburd \cite{ConreyGamburd}. Motivated by the classical problem of computing {moments} of the zeta function, they showed that $\Psi_{2q,1}(x)\sim a_q \gamma_q(\log x)^{q^2}$ when $q\in \mathbb{N}$ and $\alpha=1$, where $a_q$ is the ``arithmetic" constant in the Keating-Snaith conjecture \cite{KS} and $\gamma_q$ is the volume of a certain convex polytope. The order of magnitude $(\log x)^{q^2}$ was shown to persist to non-integer $q>0$ in \cite{BondarenkoHeapSeip, GerspachLow}.

A more nuanced picture has emerged when $\alpha> 1$: while $\Psi_{2q,\alpha}\asymp (\log x)^{(q\alpha)^2}$ for $q>1/2$ \cite{BondarenkoHeapSeip}, the order of magnitude for small $q>0$ was determined up to $(\log\log x)$ factors by Gerspach \cite{GerspachLow} to be 
\begin{align}\label{conj5.4}
    \Psi_{2q,\alpha}(x)\asymp \left\{
		\begin{array}{lll}
			(\log x)^{2(\alpha-1)q}(\log\log  x)^{O(1)} & \mbox{if } 1\leq \alpha <2 \mbox{ and }0<q\leq 2(\alpha-1)/\alpha^2 \\
			(\log x)^{(q\alpha)^2} & \mbox{if }  1\leq \alpha <2\mbox{ and }  2(\alpha-1)/\alpha^2<q\leq 1/2\\
			(\log x)^{q\alpha^2/2}(\log\log x)^{O(1)} & \mbox{if }\alpha\geq 2\mbox{ and } 0<q<1/2,
		\end{array}
	\right.
\end{align}
following initial progress in \cite{BSSZ}. In his thesis, he later conjectured the correct exponent of $(\log\log x)$ in the above, using heuristics based on work of Arguin-Ouimet-Radziwiłł \cite{AOR} and the Fyodorov-Hiary-Keating conjectures (see Conjecture 5.4 in \cite{Gerspach}). Our last result establishes the upper bound in this conjecture, in the first regime in \eqref{conj5.4}.
\begin{theorem}\label{thm:pseudomoments} Let $\alpha\in (1,2)$, $0<q<2(\alpha-1)/\alpha^2$ be fixed. Then uniformly for large $x$,
\[
    \Psi_{2q,\alpha}(x)\ll \frac{(\log x)^{2q(\alpha-1)}}{(\log\log x)^{q(3\alpha/2)}}.
\]
\end{theorem}
\noindent The proof combines Gerspach's original argument \cite{Gerspach} with the bound from Theorem \ref{thm:mainbound}. It also requires us to extend the conclusion of Theorem \ref{thm:mainbound} to integrals over \textit{mesoscopic} intervals, meaning intervals of length $\asymp(\log y)^\theta$ for $\theta\in (-1,0]$. We achieve this using a tilted measure argument.

\subsection*{Organisation} Section \ref{sec:Helson} proves  Helson's conjecture, by first reducing the claim to that of Proposition \ref{prop:2bound} (in Section \ref{sec:passing}), then proving said proposition in Sections \ref{sec:restrictedeuler} and \ref{sec:largedeviations}. In Section \ref{sec:supercritical}, we show how to adapt our approach to prove Theorem \ref{thm:mainbound}; this relies on a result on the structure of the maxima of $|F_y(\sigma+ih)|$ on $[0,1]$, proved later in Section \ref{sec:maximum}, where we also establish Corollaries \ref{cor:maximum} and \ref{cor:level}. Finally, Section \ref{sec:pseudomoments} proves Theorem \ref{thm:pseudomoments}, and the appendix compiles various Gaussian approximation estimates used throughout the paper.
 
\subsection*{Notation} We use standard asymptotic notation, writing $f(T) = O(g(T))$ or $f(T) \ll g(T)$ to mean that $\limsup_{T\to\infty} |f (T )/g(T )|$ is bounded, and $f(T) = o(g(T))$ to mean that $|f(T)/g(T)|\to 0$ as $T\to\infty$. A subscripted parameter next to $o, O$ or $\ll$ indicates that the implicit constant may depend on that parameter.
\subsection*{Acknowledgements and funding}
I thank Louis-Pierre Arguin, Seth Hardy and Mo Dick Wong for their encouragement and comments, Adam Harper for feedback on a preliminary version of the work, and Nathan Creighton for his careful reading of the current version. I also thank Maxim Gerspach for helpful conversations about pseudomoments, and Christopher Atherfold for taking interest in the work. This work is supported by the EPSRC Centre for Doctoral Training in Mathematics of Random Systems: Analysis, Modelling and Simulation (EP/S023925/1). 
\section{A short proof of Helson's conjecture}\label{sec:Helson}
This section proves the upper bound in \eqref{eq:Harper}. As observed by Harper \cite{HarperRMF}, it suffices to do so for $q\geq 2/3$, since
\begin{equation}\label{eq:holderlift}
    \mathbb{E}\bigg\{\Big|\frac{1}{\sqrt{x}}\sum_{n\leq x} f(n)\Big|^{2q}\bigg\} \leq \mathbb{E}\bigg\{\Big|\frac{1}{\sqrt{x}}\sum_{n\leq x}f(n)\Big|^{4/3}\bigg\}^{3q/2} \ll ({\log\log x})^{-q/2}
\end{equation}
by Hölder's inequality and the claim for $q=4/3$. We begin by showing that for $q\in [2/3,1]$,
\begin{equation}\label{eq:passingcritical}
    \mathbb{E}\bigg\{\Big|\frac{1}{\sqrt{x}}\sum_{n\leq x}f(n)\bigg|^{2q}\Big\} \ll \mathbb{E}\bigg\{\bigg(\frac{1}{\log y}\int_0^1 \big|F_y(1/2+ih)\big|^2\mathrm{d}h\bigg)^q\bigg\}+o\big((\log\log y)^{-q/2}),
\end{equation}
where $\log y=\log x/(\log\log x)^2$. This is the content of Section \ref{sec:passing}, which we emphasise is not new and is only included to make our proof of \eqref{eq:Harper} self-contained. The main difficulty then lies in bounding the right-hand side in \eqref{eq:passingcritical} (cf.\ Proposition \ref{prop:2bound}), which is achieved in Sections \ref{sec:restrictedeuler} and \ref{sec:largedeviations}.

\subsection{Reduction to moments of integrals of Euler products}\label{sec:passing} To prove \eqref{eq:passingcritical}, we follow the presentation of Gorodetsky and Wong \cite{GWHelson}, which streamlines Harper’s argument from \cite{HarperRMF} by incorporating a simplification later introduced in his work on character sums \cite[p.~13]{HarperDeterministic}.

By the law of total expectation and Jensen's inequality, we begin by writing
\begin{align}\label{eq:mainsum_part1}
    \mathbb{E}\bigg\{\bigg|\frac{1}{\sqrt{x}}\sum_{n\leq x}f(n)\bigg|^{2q}\bigg\} &\leq \mathbb{E}\Bigg\{\mathbb{E}\bigg\{\bigg|\frac{1}{\sqrt{x}}\sum_{n\leq x}f(n)\bigg|^{2}\,\bigg|\,\mathcal{F}_y\bigg\}^q \Bigg\},
\end{align}
where $\mathcal{F}_y$ is the  $\sigma-$algebra generated by $(Z_p)_{p\leq y}$.
Using the multiplicativity of $f$, we can decompose the partial sum of $f(n)$ up to $n\leq x$ as
\begin{align*}
    \sum_{n\leq x}f(n)=\sum_{\substack{n_1n_2\leq x\\ p|n_1 \implies p>y\\p|n_2\implies p\leq y}}f(n_1n_2)=\sum_{\substack{1\leq n_1\leq x\\p|n_1\implies p>y}} f(n_1)\sum_{\substack{n_2\leq x/n_1\\ p|n_2\implies p\leq y}} f(n_2),
\end{align*}
and use the orthogonality relation $\mathbb{E}\big\{f(n)\overline{f(m)}\big\} = \mathbf{1}(n=m)$. Note that this still holds upon conditioning on $\mathcal{F}_y$, provided $m$ and $n$ only have prime factors strictly greater than $y$. It follows that \eqref{eq:mainsum_part1} is
\begin{equation}\label{eq:mainsum_part2}
         \leq \mathbb{E}\Bigg\{\Bigg(\sum_{\substack{1\leq n_1\leq x\\p|n_1\implies p>y}} \Bigg|\frac{1}{\sqrt{x}}\sum_{\substack{n_2\leq x/n_1\\ p|n_2\implies p\leq y}} f(n_2)\Bigg|^{2}\Bigg)^q\Bigg\}.
\end{equation}

The strategy then consists of smoothing the outer summation into an integral, in order to pick up the density of integers $n_1$ which are $y$--rough (meaning $p|n_1\!\!\implies\!\! p>y$). That being said, this only yields the desired savings if $n_1$ is large enough ($> x^{3/4},$ say), and we must therefore handle the sum over smaller $n_1$ separately. 

We can separate the $q$-th moment of the sum over $1\leq n_1\leq x^{3/4}$ from the quantity in \eqref{eq:mainsum_part2} by subadditivity of $x\mapsto x^q$ and Jensen's inequality, since $q<1$. This gives 
\[
    \Bigg(\sum_{\substack{1\leq n_1\leq x^{3/4}\\p|n_1\implies p>y}} \frac{1}{x}\mathbb{E}\bigg\{\bigg|\sum_{\substack{n_2\leq x/n_1\\ p|n_2\implies p\leq y}} f(n_2)\bigg|^{2}\bigg\}\Bigg)^q=\Bigg( \sum_{\substack{1\leq n_1\leq x^{3/4}\\p|n_1\implies p>y}} \frac{\Psi(x/n_1, y)}{x}\Bigg)^{q},
\]
where $\Psi(x,y)$ counts the number of $y$-smooth numbers $n\leq x$ (meaning $p|n\!\!\implies\!\! p\leq y$). Using a well-known estimate for $\Psi$ (Theorem 5.3.1 in \cite{CojocaruMurty}), this is
\begin{align}\label{eq:error}
	\ll \bigg(\sum_{\substack{1\leq n_1 \leq x^{3/4}\\ p|n_1\implies p>y}} \frac{(\log y)^{A}}{n_1} e^{-c\tfrac{\log(x/n_1)}{\log y}}\bigg)^q\leq {(\log y)^{Aq}}e^{-cq\tfrac{\log x}{\log y}} \bigg(1+\sum_{y<n_1\leq x^{3/4}} e^{c\tfrac{\log n_1}{\log y}}/n_1\bigg)^q,
\end{align}
for a pair of absolute constants $A,c>0$. The sum in the right-hand side is bounded by
\[
    \int_{\lceil y\rceil}^{x^{3/4}+1} \frac{e^{c\log u/\log y}}{u}\mathrm{d}u\ll \int_{\log y}^{(3/4)\log x} e^{cv/\log y}\mathrm{d}v\ll (\log y) e^{c\tfrac{\log (x^{3/4})}{\log y}},
\]
and it follows that \eqref{eq:error} is $\ll (\log y)^{(A+1)q}e^{-cq(\log x^{1/4}/\log y)}=o\big((\log\log y)^{-q/2}\big)$ uniformly over $q$.

We now turn to the sum over $n_1\in (x^{3/4}, x]$ in Equation \eqref{eq:mainsum_part2}. By grouping terms according to the value $r$ of $\lfloor x/n_1\rfloor$, we can rewrite this sum as 
\[
	\sum_{1\leq r <x^{1/4}} \bigg|\frac{1}{\sqrt{r}}\sum_{\substack{n_2\leq r\\ p|n_2\implies p\leq y}} f(n_2)\bigg|^{2} \sum_{\substack{x^{3/4}< n_1\leq x\\p|n_1\implies p>y\\ \lfloor x/n_1\rfloor = r}} \frac{1}{n_1}.
\]
The inner sum over $n_1$ can now be estimated using the approximate density of $y$-rough numbers. Indeed, if we let $\Phi(x,y)$ count the number of such integers smaller than $x$, a standard sieve estimate yields
\[
	\sum_{\substack{x^{3/4}\leq n_1\leq x\\p|n_1\implies p>y\\ \lfloor x/n_1\rfloor = r}} \frac{1}{n_1} \leq \frac{\Phi(x/r,y)-\Phi(x/(r+1),y)}{x/r} \ll \frac{1}{r\cdot \log y},
\]
uniformly over $r\in [1,x^{1/4})$ (see \cite{CojocaruMurty}, Theorem 6.2.5). It follows that
\begin{align*}
    \mathbb{E}\Bigg\{\Bigg(\sum_{\substack{1\leq n_1\leq x\\p|n_1\implies p>y}} \Bigg|\frac{1}{\sqrt{x}}\sum_{\substack{n_2\leq x/n_1\\ p|n_2\implies p\leq y}} f(n_2)\Bigg|^{2}\Bigg)^q\Bigg\} &\ll \frac{1}{(\log y)^q}\mathbb{E}\Bigg\{\Bigg(\sum_{1\leq r<x^{1/4}} \Bigg|\frac{1}{r}\sum_{\substack{n_2\leq r\\ p|n_2\implies p\leq y}} f(n_2)\Bigg|^{2}\Bigg)^q\Bigg\} \\
    & \ll\frac{1}{(\log y)^q}\mathbb{E}\Bigg\{\Bigg(\int_0^{\infty} \bigg|\sum_{\substack{n_2\leq h\\ p|n_2\implies p\leq y}} f(n_2)\bigg|^{2}\frac{\mathrm{d}h}{h^2}\Bigg)^q\Bigg\},
\end{align*}
which by Parseval's theorem (in the form of Equation (5.26) in \cite{MontgomeryVaughan}) equals
\[
    \frac{1}{(\log y)^q}\mathbb{E}\Bigg\{\Bigg(\int_{\mathbb{R}} \frac{|F_{y}(1/2+ih)|^2}{|1/2+ih|^2}\mathrm{d}h\Bigg)^q\Bigg\}.
\]
Noting that $(f(p)p^{-ih},h\in [0,1])$ is equal in distribution to $(f(p)p^{-i(h+n)},h\in [0,1])$ for any fixed $n\in \mathbb{Z}$, this is
\begin{equation}\label{eq:choppingup}
     \ll \frac{1}{(\log y)^q}\mathbb{E}\bigg\{\bigg(\int_0^{1} {|F_{y}(1/2+ih)|^2}\mathrm{d}h\bigg)^q\bigg\} \bigg(\sum_{n\in \mathbb{Z}}\frac{1}{(1+n^2)^q}\bigg)
\end{equation}
 for $q\geq 2/3$. The claim follows since the sum over $n$ is bounded.
\subsection{Bounding moments of integrals of Euler products}\label{sec:restrictedeuler} 

We now turn to the proof of Proposition \ref{prop:2bound}. Without loss of generality, assume that $x>C_0$ for a fixed, large constant $C_0>0$ and set
\[
    t=t(x)=\log\log x.
\]
 We define the following second-order approximation to $\log|F_x(s)|$:
\begin{equation}\label{eq:walk}
    S_t(s)=\sum_{\
     C_0<p\leq \exp(e^t)} X_p(s),\text{ where } X_p(s):= \Re\bigg(\frac{Z_p}{p^{s}}+\frac{Z_p^2}{2p^{2s}}\bigg).
\end{equation}
This choice of notation reflects our intention to view $S_t(1/2+ih)$ (and hence $\log |F_x(1/2+ih)|$) as a random walk with $t$ increments $S_{j}(1/2+ih)-S_{j-1}(1/2+ih)$ of variance
\[
   \mathbb{E}\Big\{\big(S_{j}(1/2+ih)-S_{j-1}(1/2+ih)\big)^2\Big\}= \frac{1}{2}\sum_{\exp(e^{j-1})<p\leq \exp(e^j)} \Big(\frac{1}{p}+\frac{1}{4p^2}\Big),
\]
which is roughly $1/2$ by the prime number theorem. 

Noting that
\begin{equation}\label{eq:trimming}
        {|F_{x}(1/2+ih)|^2}\ll e^{2S_{t}(1/2+ih)+\sum_{p}\sum_{k\geq 3}p^{-k/2}} \leq e^{2S_t(1/2+ih)+4\sum_p p^{-3/2}} \ll e^{2S_t(1/2+ih)},
\end{equation}
for any $h\in [0,1]$,
it suffices to show that for any $q\in [0,1]$,
\begin{equation}\label{eq:helsonmoment}
    \mathbb{E}\big\{\mathcal{Z}_2^q\big\}\ll \big((1-q)\sqrt{t}+1\big)^{-q},\quad \text{ where } \mathcal{Z}_{2}:= \frac{1}{e^t}\int_0^1 e^{2 S_t(1/2+ih)}\mathrm{d}h.
\end{equation}
Furthermore, we may assume that $q<1-1/\sqrt{t}$; the desired bound is simply $\ll 1$ otherwise, in which case it follows directly from  Hölder's inequality and the Laplace transform estimate in Lemma \ref{lem:LDestimates}:
\begin{equation}\label{eq:trivialcritical}
    \mathbb{E}\big\{\mathcal{Z}_2^q\big\} \ll {\mathbb{E}\big\{e^{2S_t(1/2+ih)}\big\}^q} e^{-tq}\ll 1.
\end{equation}

\bigskip 

Assuming that $q\in [0,1-1/\sqrt{t}]$, we now proceed in two steps. The first will be to estimate the expectation of
\[
     \mathcal{Z}_2(A):=\frac{1}{e^t}\int_0^1e^{2 S_{t}(1/2+ih)}\mathbf{1}(h\in G_A)\mathrm{d}h,
\]
where $G_A\subseteq [0,1]$ is a suitably chosen set of ``good points" that we define shortly (cf.\ Proposition \ref{prop:boundZ(A)}). We then leverage the fact that most points $h$ are in $G_A$ with high probability (cf.\ Lemma \ref{lem:goodevent}) to upgrade this estimate to \eqref{eq:helsonmoment}. To define $G_A$, we make the observation that $(S_t(1/2+ih))_{h\in [0,1]}$ is approximately a \textit{logarithmically-correlated field}; that is, $S_t(1/2+ih)$ is approximately Gaussian for each $h$ (cf. Lemma \ref{lem:berryesseen}), and 
 \[
     \mathbb{E}\big\{S_t(1/2+ih)S_t(1/2+ih')\big\}\approx\frac{1}{2}\log\frac{1}{|h-h'|\lor e^{-t}}.
 \]
  (This can be made precise by a straightforward application of a quantitative prime number theorem.) The extremal statistics of such stochastic processes have been extensively studied \cite{ArguinSurvey}, beginning with the pioneering work of Bramson on branching Brownian motion \cite{Bramson}. Using a similar approach, we will prove in Lemma \ref{lem:goodevent} that for large $t$,
  \begin{equation}\label{eq:}
     \mathbb{P}\Big(\max_{h\in [0,1]} S_t(1/2+ih) \leq m(t)+A\Big)\to 1 \text{ as $A\to\infty$},\quad \text{where } m(t):=t-\frac{3}{4}\log t.
  \end{equation}
  To be precise, we will argue that the paths $j\mapsto S_j(1/2+ih)$ for which $S_t(1/2+ih)\approx A+m(t)$ typically display linear growth, while remaining under a \textit{barrier} $m(j)+B(j)+A$ at each time $j$. Crucially, we can pick $B(j)$ to be \textit{smaller} than the typical fluctuations of a Brownian bridge from $A$ to $m(t)+A$, which ultimately yields a logarithmic correction in the order of the maximum (an additional $-\tfrac{1}{2}\log t$).

For Proposition \ref{prop:2bound}, we only require a weak version of this fact where $B(j)$ is relatively large. Let
\[
    G_{A,\sigma}=\{h\in [0,1]: S_j(\sigma+ih)\in [L_A(j),U_A(j)],  A/4\leq j\leq t \}
\]
for any $\sigma\in \mathbb{R}$ and $A>4C_0$, where 
 \begin{equation}\label{eq:upperlowerbarrier}
    U_A(j)=A+j+2\log\!\big(1+j\land (t-j)\big),\quad L_A(j)=A-20j,\quad \forall A/4\leq j\leq t,
\end{equation}
and $U_A(j)=-L_A(j)=\infty$ for smaller $j$. The lower barrier $L_A$ is needed later in the proof, when approximating $S_t$ by a bona fide Gaussian random walk. In this section, we only consider $\sigma=1/2$ and therefore omit the dependence on $\sigma$, writing $G_A=G_{A,1/2}$.
 \begin{proposition}\label{prop:boundZ(A)} Uniformly in all large $t$ and $4C_0<A\leq  3\sqrt{t}$, $\mathbb{E}\big\{\mathcal{Z}_2(A)\big\}\ll {A}/{\sqrt{t}}$.
\end{proposition}
\begin{proof}
    We express the integral on the right-hand side in terms of the {large deviation frequencies} of $S_t$. Letting
\begin{equation}\label{eq:defmeas}
    \mathcal{S}(V,G_A):=\text{meas}\{h\in [0,1]:S_{t}(1/2+ih)>V, h\in G_A\},
\end{equation}
for any $V>0$, we can rewrite $\mathbb{E}\{\mathcal{Z}_2(A)\}$ using Fubini's theorem as
\begin{equation}\label{eq:LDpartition}
    \mathbb{E}\bigg\{\frac{1}{e^t}\int_0^1 \int_{-\infty}^\infty 2 e^{2 V}\mathbf{1}(h\in G_A, S_{t}(1/2+ih)>V)\mathrm{d}V\mathrm{d}h\bigg\}=\frac{1}{e^t}\int_{-\infty}^{U_A(t)}2 e^{2 V}\mathbb{E}\big\{\mathcal{S}(V,G_A)\big\}\mathrm{d}V.
\end{equation}
To bound $\mathbb{E}\mathcal{S}(V,G_A)$, we derive the following bound on the large deviation probability of $S_{t}(1/2+ih)$ when $h\in G_A$. Informally, it states that this probability is comparable to that of a Brownian bridge of length $t$, from 0 to $V$, remaining under $U_A(s)$ at all times $s\in [0,t]$. For later use, we state a more general result which applies to $S_t(\sigma+ih)$ for $\sigma$ near $1/2$, and another choice of envelope $U_A(s)$. The proof is technical and will be given in Section \ref{sec:largedeviations}.

\begin{lemma}\label{lem:largedeviations} Let $h\in[0,1]$ and $x$ be large. Then there exists an absolute constant $C>0$ such that the following holds. Fix $\sigma_k=1/2-k/\log x$ and $0\leq k\leq \log\log\log x$. For any $A/4<t\leq \log\log x-k$, $4C_0<A\leq 3\sqrt{t}$ and $0\leq V\leq U_A(t)$, 
\[
    \mathbb{P}\big(S_{t}(\sigma_k+ih)>V, h\in {G_{A,\sigma_k}}\big)\leq C\cdot \frac{A\big(U_A(t)-V+C\big)}{t}\frac{e^{-V^2/t}}{\sqrt{t}}.
\]
Furthermore, the same bound holds upon replacing $G_{A,\sigma_k}$ and $U_A$ by $G_{A,\sigma_k}^{\mathrm{max}}$ and $U_A^{\mathrm{max}}$ (defined in Equations \eqref{def:Gmax} and \eqref{def:Umax}, respectively).
\end{lemma}
\begin{proof}
    See Section \ref{sec:largedeviations}.
\end{proof}

Noting that $S_t(1/2+ih)$ is equal in distribution to $S_t(1/2)$ for any $h\in [0,1]$, Fubini's theorem yields
\[
    \mathbb{E}\big\{\mathcal{S}(V,G_A)\big\}=\int_0^1 \mathbb{P}\big(S_t(1/2+ih)>V, h\in G_A\big)\mathrm{d}h=\mathbb{P}\big(S_{t}(1/2)>V,0\in G_A\big), 
\]
and it follows by Lemma \ref{lem:largedeviations} that 
\begin{align*}
    \mathbb{E}\big\{\mathcal{Z}_2 (A)\big\}&\ll \frac{1}{e^t}\int_{-\infty}^0 e^{2V} \mathrm{d}V+\frac{1}{e^t}\int_{0}^{U_A(t)}\frac{A\big(U_A(t)- V+C\big)}{t}\frac{e^{2V-V^2/t}}{\sqrt{t}}\mathrm{d}V,\\
    &\ll O(e^{-t})+\frac{1}{e^t}\int_{0}^{A+t}\frac{A\big(A+t- V+C\big)}{t}\frac{e^{2V-V^2/t}}{\sqrt{t}}\mathrm{d}V.
\end{align*}
By the change of variables $u=({t}-V)/\sqrt{t}$, the remaining integral is bounded by
\begin{align*}
    \ll \frac{A}{\sqrt{t}}\int_{-A/\sqrt{t}}^{\sqrt{t}} (3+o(1)+|u|/2)e^{-u^2}\mathrm{d}u 
    &\ll \frac{A}{\sqrt{t}}\int_{0}^\infty (1+u)e^{-u^2}\mathrm{d}u \ll \frac{A}{\sqrt{t}}.\qedhere
\end{align*}
\end{proof}

\begin{lemma}\label{lem:goodevent} Uniformly in all large $t$ and $4C_0<A\leq 3 \sqrt{t}$, $\mathbb{P}\big(\exists h\notin G_A \big) \ll e^{-2A}$.
\end{lemma}
\begin{proof} 
    A union bound over $j\leq t$ yields
    \begin{equation}\label{eq:unionbound}
        \mathbb{P}\big(\exists h\notin G_A \big)\leq  \sum_{A/4\leq j\leq t} \mathbb{P}\big(\exists h: S_j(1/2+ih)>U_A(j)\big)+\sum_{A/4\leq j\leq t}\mathbb{P}\big(\exists h: S_j(1/2+ih)<L_A(j)\big)
    \end{equation}
    For each $j\leq t$, let $(I_j)_j$ be a partition of $[0,1]$ into disjoint intervals $I_j$ of width $e^{-j}$. Using the fact that $(S_j(1/2+ih), h\in I_j)\overset{d}{=}(S_j(1/2+ih), h\in [0,e^{-j}))$ for each $I_j$, another union bound yields
    \[
        \mathbb{P}\big(\exists h: S_j(1/2+ih)>U_A(j)\big) \leq \sum_{I_j} \mathbb{P}\big(\max_{h\in I_j} S_j(1/2+ih)>U_A(j)\big) = e^j \mathbb{P}\big(\max_{h\in [0,e^{-j})} S_j(1/2+ih)>U_A(j)\big).
    \]
    We then claim that $\max_{h\in [0,e^{-j})}S_j(1/2+ih)$ is comparable in law to $S_j(1/2)$. This is made rigorous by Lemma \ref{lem:chaining}, which we prove using a standard chaining argument in the appendix, and by which
    \begin{equation}\label{eq:maxcomplement}
        e^j\cdot \mathbb{P}\big(\max_{h\in [0,e^{-j})} S_j(1/2+ih)>U_A(j)\big) \ll \exp\big(j-U_A(j)^2/j\big) \ll e^{-2A} \big(1+j\land (t-j)\big)^{-2},
    \end{equation}
    uniformly in $j$. For the second sum in \eqref{eq:unionbound}, we simply note that
    \begin{equation}\label{eq:mincomplement}
        \mathbb{P}\big(\exists h: S_j(1/2+ih)<L_A(j)\big)= \mathbb{P}\big(\exists h: -S_j(1/2+ih)>20j-A\big),
    \end{equation}
    and that the bound in Lemma \ref{lem:chaining} applies to $\max_{h\in [0,e^{-j})}-S_j(1/2+ih)$ (cf.\ Remark \ref{rem:symmetry}). We can therefore use the same argument  used to bound $\mathbb{P}(\exists h:S_j(1/2+ih)>U_A(j))$, which in this case yields $\ll e^{j-400j+40A}\ll e^{-2A}j^{-2}$ provided $j\geq A/4$. Using this in \eqref{eq:unionbound} along with the estimate in \eqref{eq:maxcomplement} and summing over $j$ proves the lemma.
\end{proof}

 \begin{proof}[Proof of Proposition \ref{prop:2bound}]
     We use an interpolation argument of Soundararajan and Zaman \cite{SoundZaman}. Consider the sequence $(n_j)_{j\geq 1}$, defined by $n_j:= \frac{j}{(1-q)}$. Note that if $4C_0+1=:j_0\leq j\leq \lceil(1-q)\sqrt{t}\rceil$, $4C_0<n_j\leq 3 \sqrt{t}$ for large $t$, and in particular, Lemma \ref{lem:goodevent} and Proposition \ref{prop:boundZ(A)} apply with $A=n_j$.

Using the fact that $\mathbf{1}([0,1]\subseteq  G_{n_j})\leq \mathbf{1}(h_0\in G_{n_j})$ for any fixed $h_0$, we decompose $\mathcal{Z}_2$ iteratively as:
\begin{align}\label{eq:decomposition}
    \mathcal{Z}_2\leq \mathcal{Z}_2(n_{j_0})+\sum_{j_0\leq j\leq K} \mathbf{1}(\exists h\notin G_{n_j})\mathcal{Z}_2(n_{j+1})+\mathbf{1}(\exists h\notin G_{n_{K+1}})\mathcal{Z}_2
\end{align}
where $K=j_0+\lceil \sqrt{t}(1-q)\rceil$.
Taking the $q$-th moment of both sides and using the subadditivity of $x\mapsto x^q$, 
\begin{align*}
    \mathbb{E}\big\{\mathcal{Z}_2^q\big\} \leq \mathbb{E}\big\{\mathcal{Z}_2(n_{j_0})\big\}^q+ \sum_{j_0\leq j\leq K}\mathbb{E}\big\{\mathbf{1}(\exists h\notin G_{n_j})\mathcal{Z}_2(n_{j+1})^q\big\}+\mathbb{E}\big\{\mathbf{1}(\exists h\notin G_{n_{K+1}})\mathcal{Z}_2^q\big\},
\end{align*}
which by Hölder's inequality and Lemma \ref{lem:goodevent} is
\begin{align*}
&\ll\mathbb{E}\big\{\mathcal{Z}_2(n_{j_0})\big\}^q+ \sum_{j_0\leq j\leq K}\mathbb{P}(\exists h\notin G_{n_j})^{1-q}\mathbb{E}\big\{\mathcal{Z}_2(n_{j+1})\big\}^q+\mathbb{P}(\exists h\notin G_{n_{K+1}})^{1-q}\mathbb{E}\big\{\mathcal{Z}_2\big\}^q,\\
&\ll \mathbb{E}\big\{\mathcal{Z}_2(n_{j_0})\big\}^q+ \sum_{j_0\leq j\leq K} e^{-2j}\mathbb{E}\big\{\mathcal{Z}_2(n_{j+1})\big\}^q+e^{-2(K+1)}\mathbb{E}\big\{\mathcal{Z}_2\big\}^q.
\end{align*}
Using Proposition \ref{prop:boundZ(A)} and the estimate in \eqref{eq:trivialcritical}, we conclude that
\[
     \mathbb{E}\big\{\mathcal{Z}_2^q\big\} \ll e^{-2(K+1)}+\big(\sqrt{t}(1-q)\big)^{-q} \sum_{j\leq K} (j+1)^qe^{-2j} \ll \big(\sqrt{t}(1-q)\big)^{-q}.\qedhere
\]
 \end{proof}
\subsection{Proof of Lemma \ref{lem:largedeviations}}\label{sec:largedeviations} We use a discretisation argument inspired by that of \cite{FHK1} (Section 7), albeit in a simpler setting. A similar idea was used by Harper in \cite{HarperRMF}. To alleviate the notation, we will assume without loss of generality that $h=0$ and write $S_j=S_j(\sigma_k)$ for the remainder of this section.

Fix $r=\lceil A/4\rceil$, and let $Y_j:=S_j-S_{j-1}$ denote the $j$-th increment of $S_{t}$ when $j>r$. It will be helpful to abuse notation by defining $Y_r:=S_{r}$. To discretise the range of the $(Y_j)_{r\leq j\leq t}$, let $\mathcal{T}\subseteq \mathbb{R}^{t-r+1}$ denote the set of all (disjoint) tuples $(u_{r},...,u_{t})$ with $u_j\in \Delta_j\mathbb{Z}$, where $\Delta_j=j^{-4}$ for $j>r$ and $\Delta_r=1$. This mesh size ensures that $\sum_j\Delta_j<2$.

We begin with the following trivial inclusion
\begin{align}\label{eq:inclusion}
    \big\{ S_{t}\in [w,w+1),0\in G_A\big\} \subseteq \bigcup_{\mathcal{T}} \big\{Y_j\in[u_j, u_j+\Delta_j)\big\}\cap\big\{ S_{t}\in [w,w+1), 0\in G_A\big\}.
\end{align}
Furthermore, by definition of $G_A$, any $\mathbf{u}=(u_j)_j\in \mathcal{T}$ for which the intersection on the right-hand side is non-empty must necessarily satisfy the following constraints:
\begin{align}\label{eq:constraint1}
    w-2\leq w-\sum_j \Delta_j\leq &\sum_{j=r}^{t} u_j \leq w+1\\\label{eq:constraint2}
    L_A(\ell)-2\leq L_A(\ell)-\sum_j \Delta_j\leq&\sum_{j=r}^{\ell}u_j\leq U_A(\ell),\quad \forall r\leq \ell\leq t.
\end{align}
This also forces $|u_j|\leq 100j$ for each $j\in [r,t]$. 
Letting $\mathcal{T}'(w)$ be the set of all such $\mathbf{u}$, it follows that one can replace $\mathcal{T}$ by $\mathcal{T}'(w)$ in \eqref{eq:inclusion}. By a union bound and the fact that the  $Y_j$ are independent, we therefore conclude that
\begin{equation}\label{eq:tocompare}
    \mathbb{P}\big(S_{t}\in [w,w+1), 0\in G_A \big)\leq \sum_{(u_j)_j\in \mathcal{T}'(w)} \prod_{j=r}^{t}\mathbb{P}\big(Y_j\in [u_j,u_j+\Delta_j)\big).
\end{equation}
We now compare the probabilities on the right-hand side to a Gaussian counterpart. Namely, let $(\mathcal{N}_j)_{j\geq 1}$ denote a sequence of real, independent centered Gaussian random variables of variance $1/2$ when $j\leq r$, and $V_j(\sigma_k)=\sum_{\exp(e^{j-1})<p\leq \exp(e^j)}\frac{1}{2p^{2\sigma_k}}+\frac{1}{8p^{4\sigma_k}}$ otherwise. Define the random walk $\mathcal{G}_\ell=\sum_{1\leq j \leq \ell}\mathcal{N}_j$ for $\ell\geq 1$. Beginning with the $j=r$ term in \eqref{eq:tocompare}, we can use the estimate in Lemma \ref{lem:LDestimates} to get
\begin{align}\label{eq:first_increment}
    \mathbb{P}\big(Y_r\in [u_r,u_r+1)\big)\ll \frac{e^{-u_r^2/r}}{\sqrt{r}}\asymp\mathbb{P}\big(\mathcal{G}_r\in [u_r,u_r+1)\big).
\end{align}
 For $j>r$, the more precise estimate in Lemma \ref{lem:berryesseen} yields
\begin{align*}
    \mathbb{P}\big(Y_j\in [u_j,u_j+\Delta_j)\big)&=\mathbb{P}\big(\mathcal{N}_j\in [u_j,u_j+\Delta_j)\big)+O\big(e^{-ce^{j/2}}\big)\\
    &=\mathbb{P}\big(\mathcal{N}_j\in [u_j,u_j+\Delta_j)\big)\cdot\big(1+O(j^{-3})\big),
\end{align*}
where we have used the fact that $|u_j|\leq 100j$ to make the error multiplicative. Since $\prod_{j}(1+j^{-3})<\infty$, 
\begin{equation*}
    \mathbb{P}\big(S_{t}\in [w,w+1), h\in G_A \big) \ll \sum_{(u_j)_j\in \mathcal{T}'(w)}\mathbb{P}\big(\mathcal{G}_r\in [u_r,u_r+\Delta_r)\text{ and }\mathcal{N}_j\in[u_j,u_j+\Delta_j)\,\forall r<j\leq t\big).
\end{equation*}
We now undo the discretisation. For any $(u_j)_j\in \mathcal{T}'(w)$, the event in the right-hand side implies that
\[
    \forall r\leq \ell\leq t,\quad \mathcal{G}_\ell \leq U_A(\ell)+\sum_{j\leq \ell} \Delta_j\leq U_A(\ell)+2, \text{ and } |\mathcal{G}_{t}-w|\leq 1+\sum_{j=r}^t \Delta_j.
\]
By summing over $(u_j)_{j}\in\mathcal{T}'(w)$, we conclude that
\begin{equation}\label{eq:Gaussiancomparisonfinal}
    \mathbb{P}\big(S_{t}\in [w,w+1), 0\in G_A \big)\ll \mathbb{P}\big(\mathcal{G}_{t}>w-2, \mathcal{G}_{\ell}\leq U_A(\ell)+2,\forall  \ell\leq t\big).
\end{equation}
By Lemma \ref{lem:PNT}, $r/2+\sum_{j=r+1}^tV_j(\sigma_k)=t+O(1)$ uniformly in $0\leq k\leq \log\log\log x$. When $k=0$, $\sigma_k=1/2$ and $V_j(\sigma_k)={1}/{2}+O(e^{-c\sqrt{e^{r-1}}})$ by the same lemma. Otherwise,
\[
    V_j(\sigma_k)=\frac{1}{2}\Big(\mathrm{Ei}(ve)-\mathrm{Ei}(v)\Big)+O\Big(e^{-c\sqrt{e^{r-1}}}+\sum_{n>\exp(e^{r-1})}\frac{1}{n^{7/4}}\Big),\text{ where } v:=\frac{2k}{\log x}e^{j-1}\leq \frac{2}{e},
\]
uniformly in $r<j\leq t$.
The error term is contained in $[-1/4,1/4]$ by taking a larger $C_0$ if need be (recall that $r\geq C_0$), while the main term equals $\frac{1}{2}\int_{v}^{ev} (e^t/t) \mathrm{d}t\in [1/2,e^2/2]$ since $v\in (0,2/e)$. In both cases, we can use Proposition \ref{cor:Hittingprob} with $\kappa=1/4$ to estimate \eqref{eq:Gaussiancomparisonfinal}, which yields 
\[
    \ll \frac{A\big(U_A(t)-w+C\big)}{t}\frac{e^{-w^2/t}}{\sqrt{t}}
\]
for some $C>0$.
By partitioning the event $\{S_{t}(h)>V\}$ according to $S_{t}(h)-V\in [v,v+1)$ for $v\in \mathbb{Z}\cap [0,U_A(t)-V]$, we conclude that
\begin{align*}
    \mathbb{P}\big(S_{t}>V, 0\in G_A \big)\ll \frac{e^{-V^2/t}}{t^{3/2}}\sum_{v} {A\big(U_A(t)-V+C+v\big)}e^{-2v\alpha} \ll \frac{e^{-V^2/t}}{t^{3/2}}{A\big(U_A(t)-V+C'\big)}
\end{align*}
for some constant $C'> 0$.

\section{Proof of Theorem \ref{thm:main}}\label{sec:supercritical}

We now prove Theorem \ref{thm:main}. As before, the first step is a reduction to moments of an Euler product integral, carried out along the same lines as in Section \ref{sec:passing}. Owing to the additional technical complications introduced by the divisor function, we omit the details and adopt the corresponding reduction from \cite{Gerspach}. By elementary manipulations, this reduces the claim to that of Theorem \ref{thm:mainbound} as we now show.

\begin{proof}[Proof of Theorem \ref{thm:main} assuming Theorem \ref{thm:mainbound}]
By Proposition 3.6 in \cite{Gerspach}, 
\[
    \mathbb{E}\bigg\{\bigg|\frac{1}{\sqrt{x}}\sum_{n\leq x}d_\alpha(n)f(n)\bigg|^{2q}\bigg\}\ll \sum_{0\leq k\leq K} \mathbb{E}\bigg\{\bigg(\frac{1}{\log x}\int_{\mathbb{R}}\frac{|F_{x^{e^{-k}}}(\tfrac{1}{2}-\tfrac{k}{\log x}+it)|^{2\alpha}}{|\tfrac{1}{2}-\tfrac{k}{\log x}+it|^2}\bigg)^q\bigg\}+1,
\]
where $K=\lfloor \log\log\log x\rfloor$ for simplicity. Using subadditivity of $x\mapsto x^q$ and the rotational invariance in law of the $f(p)$ (as in Equation \eqref{eq:choppingup}), the right-hand side is
\begin{equation}\label{eq:supercriticalmoments}
    \ll \sum_{0\leq k\leq K} \mathbb{E}\bigg\{\bigg(\frac{1}{\log x}\int_{0}^1{\big|F_{x^{e^{-k}}}(\tfrac{1}{2}-\tfrac{k}{\log x}+ih)\big|^{2\alpha}}\bigg)^q\mathrm{d}h\bigg\}+1,
\end{equation}
provided $q>1/2$. We are free to pick such a $q$ since $(1/2,1/\alpha)$ is non-empty, and the claimed bound for any smaller $q'<q$ follows from the bound at $q$ by Hölder's inequality.

By Theorem \ref{thm:mainbound} (with $y=x^{e^{-k}}$, $\gamma=2\alpha$ and $\sigma=1/2-k/\log x$), the quantity in \eqref{eq:supercriticalmoments} is bounded by
\[
    \ll \Big(\sum_{k\geq 0} e^{-kq(2(\alpha-1)+1)}\Big)\frac{(\log x)^{2q(\alpha-1)}}{(\log\log x)^{q(3\alpha/2)}(1-\alpha q)+1}.\qedhere
\]
\end{proof}
\subsection{Proof of Theorem \ref{thm:mainbound}} Fix $0\leq k\leq  \lfloor \log\log\log x\rfloor$, and let $\sigma_k=\frac{1}{2}-\frac{k}{\log x}$. Letting $y$ be as in the statement of the theorem, we once again take
\[
    t=t(y)=\log\log y, 
\]
and note that $ t\leq \log\log x-k$ by the assumption on $y$. We also note that we can make $y$ (and $t$) large if need be without loss of generality.

The strategy follows that of Proposition \ref{prop:2bound}, with two main differences. The first is that a much more precise upper barrier is required. For a given $t$, we pick
\begin{equation}\label{def:Umax}
        U_A^{\text{max}}(j):=A+j\Big(1-\frac{3}{4}\frac{\log t}{t}\Big)+\mathcal{C}\cdot \log\big(1+j\land (t-j)\big),\quad (j\geq A/4)
\end{equation}
where $\mathcal{C}$ is any large enough constant (say, $\mathcal{C}=10^3$), and $ U_A^{\text{max}}(j)=\infty$ for smaller $j$. We expect $\max_{h\in [0,1]} S_j(\sigma_k+ih)$ to fluctuate around $j-\frac{3}{4}\log j$, $ U_A^{\text{max}}(j)$ allows it do so within a logarithmic bump, while forcing $S_t(\sigma_k+ih)$ not to exceed $t-\frac{3}{4}\log t$ by more than $A$. We will need a version of Lemma \ref{lem:goodevent} for the good set
\begin{equation}\label{def:Gmax}
        G_{A,\sigma_k}^{\text{max}}=\{h\in [0,1]: S_j(\sigma_k+ih)\in [L_A(j),U_A^{\text{max}}(j)],\,\forall j\leq t \},
\end{equation}
the proof of which is rather involved and thus postponed to Section \ref{sec:maximum}.
\begin{proposition}\label{prop:maximum} Uniformly in all large $t$ and $4C_0<A\leq t/\log t$, $\mathbb{P}(\exists h\notin G_{A,\sigma_k}^{\mathrm{max}})\ll Ae^{-2A-A^2/t}$.
\end{proposition}
\begin{proof}
    See Section \ref{sec:maximum}.
\end{proof}
The second way in which the proof differs from that of Proposition \ref{prop:2bound} is that if we let
\[
    \mathcal{Z}_{\gamma,\sigma_k}:= \frac{1}{e^t}\int_0^{1}e^{\gamma S_t(\sigma_k+ih)} \mathrm{d}h,
\]
then the trivial bound for $\mathbb{E}\{\mathcal{Z}_{\gamma,\sigma_k}^q\}$ is not sharp in the leading order when $\gamma>2$. Indeed, while Fubini's theorem and a Laplace transform estimate (Lemma \ref{lem:LDestimates}) yield
\begin{equation}\label{eq:trivialsuper}
    \mathbb{E}\big\{\mathcal{Z}_{\gamma,\sigma_k}^q\}\leq \mathbb{E}\big\{ \mathcal{Z}_{\gamma,\sigma_k}\big\}^q\ll e^{qt{(\gamma^2/4-1})},
\end{equation}
for $q<2/\gamma$, we will show that the exponent in the right-hand side can be brought down to $qt(\gamma-2)$. This will replace the trivial bound in an interpolation argument similar to the one in Section \ref{sec:Helson}.

\begin{lemma}\label{lem:crudesupercritical} Let $\gamma\in (2,4), q\in [0,2/\gamma]$. Then uniformly in $k\leq \lfloor \log\log\log x\rfloor$ and $t\leq \log\log x-k$, $\mathbb{E}\big\{\mathcal{Z}_{\gamma,\sigma_k}^q\big\} \ll e^{qt(\gamma-2)}$.
\end{lemma}
\begin{proof}
    For any $A\geq 1$, let $E_A:=\{\max_{h\in [0,1]}S_t(\sigma_k+ih)\leq t+A\}$. By a union bound and Lemma \ref{lem:chaining}, 
    \[
\mathbb{P}\big(\lnot E_A\big) \leq e^t \cdot \mathbb{P}\bigg(\max_{h\in[0,e^{-t}]} S_t(\sigma_k+ih)>t+A\bigg) \ll e^{-2A}
    \]
    uniformly in $A\geq 1$. Furthermore, Fubini's theorem and Lemma \ref{eq:density} yield
    \begin{align*}
        \mathbb{E}\Big\{\mathcal{Z}_{\gamma,\sigma_k} \mathbf{1}(E_A)\Big\} &\ll \frac{1}{e^t}\int_{-\infty}^{t+A}e^{\gamma V}\mathbb{P}\big(S_t(\sigma_k)>V\big)\mathrm{d}V \\ &\ll \frac{1}{e^t}\int_{-\infty}^0 e^{\gamma V}\mathrm{d}V+\frac{1}{e^t} \int_{0}^{t+A} \frac{e^{\gamma V-V^2/t}}{\sqrt{t}}\mathrm{d}V\ll e^{(\gamma-2)t}
    \end{align*}
    uniformly in $A\leq (\tfrac{\gamma}{2}-1)t$.
    Using the decomposition
    \[
            \mathcal{Z}_{\gamma,\sigma_k}\leq \mathcal{Z}_{\gamma,\sigma_k}\mathbf{1}(E_2)+\sum_{2\leq j\leq J} \mathbf{1}(\lnot E_j)\big(\mathcal{Z}_{\gamma,\sigma_k}\mathbf{1}(E_{j+1})\big)+\mathbf{1}(\lnot E_{J+1})\mathcal{Z}_{\gamma,\sigma_k},
    \]
    where $J=\lfloor (\gamma/2-1)t\rfloor$,
    we apply Hölder's inequality and the subadditivity of $x\mapsto x^q$ to conclude that
    \begin{align*}
        \mathbb{E}\big\{\mathcal{Z}_{\gamma,\sigma_k}^q\big\} 
            &\leq \mathbb{E}\Big\{\mathcal{Z}_{\gamma,\sigma_k} \mathbf{1}(E_2)\Big\}^q +\sum_{2\leq j\leq J}\mathbb{P}\big(\lnot E_j\big)^{1-q}\mathbb{E}\Big\{\mathcal{Z}_{\gamma,\sigma_k} \mathbf{1}(E_{j+1})\Big\}^q+\mathbb{P}(\lnot E_{J+1})^{1-q}\mathbb{E}\big\{\mathcal{Z}_{\gamma,\sigma_k}\big\}^q\\
        &\ll e^{(\gamma-2)tq}\Big(\sum_{j\leq J} e^{-2j(1-q)}\Big)+ e^{(\gamma-2)tq}e^{-2(\gamma/2-1)t+qt(\gamma^2/4-1)} \ll e^{(\gamma-2)tq}.  \qedhere
    \end{align*}
\end{proof}

Armed with Proposition \ref{prop:maximum} and Lemma \ref{lem:crudesupercritical}, the proof of Theorem \ref{thm:mainbound} is essentially the same as that of Theorem \ref{prop:2bound}.
\begin{proof}[Proof of Theorem \ref{thm:mainbound}]  Assume without loss of generality that $(\gamma/2) q\in [0,1-t^{-3\gamma/4}]$, since the claim otherwise follows from Lemma \ref{lem:crudesupercritical}. Define
\[
    \mathcal{Z}_{\gamma,\sigma_k}(A):=\frac{1}{e^t}\int_0^1 e^{\gamma S_t(\sigma_k+ih)}\mathbf{1}\big(h\in G_{A,\sigma_k}^{\text{max}}\big)\mathrm{d}h.
\]
Proceeding as in Equation \eqref{eq:LDpartition}, Fubini's theorem yields
\begin{align*}
    \mathbb{E}\mathcal{Z}_{\gamma,\sigma_k}(A) &\ll O(e^{-t})+\mathbb{E}\bigg\{ \frac{1}{e^t}\int_0^1\int_0^{U_A^{\text{max}}(t)} \frac{e^{\gamma V-V^2/t}}{\sqrt{t}}\mathbf{1}\big(S_t(\sigma_k+ih)>V, h\in G_{A,\sigma_k}^{\text{max}}\big)\mathrm{d}V\mathrm{d}h\bigg\}
    \\ &\ll O(e^{-t})+ \frac{1}{e^t}\int_0^{U_A^{\text{max}}(t)} \frac{e^{\gamma V-V^2/t}}{\sqrt{t}}\mathbb{P}\big(S_t(\sigma_k)>V, h\in G_{A,\sigma_k}^{\text{max}}\big)\mathrm{d}V.
\end{align*}
By Lemma \ref{lem:largedeviations}, this integral is bounded by a constant times
\begin{align*}
    &\frac{1}{e^t}\int_0^{U_A^{\text{max}}(t)} \frac{e^{\gamma V-V^2/t}}{\sqrt{t}}\frac{A\big(U_A^{\text{max}}(t)-V+C\big)}{t}\mathrm{d}V,
\end{align*}
    which, by the substitution $u=U_A^{\text{max}}(t)-V$, is bounded by
\begin{align}\label{eq:supercritical_ongood}
   &\ll A \frac{e^{\gamma U_A^{\max}(t)-U_A^{\max}(t)^2/t}}{t^{3/2}}\frac{1}{e^t}\int_0^\infty (u+C)e^{(2U_A^{\max}(t)/t-\gamma)u-u^2/t}\mathrm{d}u\nonumber  \\
   &\ll \big(Ae^{(\gamma-2)A}\big)\frac{e^{(\gamma-2)t}}{t^{3\gamma/4}}\cdot \int_{0}^\infty (u+C) e^{(2-\gamma+o(1))u}\mathrm{d}u \ll \big(Ae^{(\gamma-2)A}\big)\frac{e^{(\gamma-2)t}}{t^{3\gamma/4}},
\end{align}
uniformly in $A$.
We are now equipped to bound $\mathbb{E}\big\{\mathcal{Z}_{\gamma,\sigma_k}^q\big\}$. Using the decomposition in \eqref{eq:decomposition} with $n_j=j/(2-\gamma q)$ and $j_0=4C_0+1$, we can bound $
    \mathbb{E}\big\{\mathcal{Z}_{\gamma,\sigma_k}^q\big\}$ by
\begin{align*}
\leq \bigg(\mathbb{E}\big\{\mathcal{Z}_{\gamma,\sigma_k}(n_{j_0})\big\}^q+\sum_{j_0\leq j\leq K}\mathbb{P}\big(\exists h\notin G_{n_j,\sigma_k}^{\max{}}\big)^{1-q} \mathbb{E}\big\{\mathcal{Z}_{\gamma,\sigma_k}(n_{j+1})\big\}^q\bigg)\!+\mathbb{E}\big\{\mathbf{1}\big(\exists h\notin G_{n_{K+1},\sigma_k}^{\text{max}}\big)\mathcal{Z}_{\gamma,\sigma_k}^q\big\}
\end{align*}
for any $K\geq j_0$. We pick $K=\lceil2\log (t^{3q\gamma/4}(2-\gamma q))\rceil+j_0$, so that $4C_0<n_j\leq t/\log t$ for each $j_0\leq j\leq K+1$. Proposition \ref{prop:maximum} and the bound in Equation \eqref{eq:supercritical_ongood} imply that the terms in brackets are
\[
    \ll \frac{e^{tq(\gamma-2)}}{t^{3\gamma q/4}(2-\gamma q)}\sum_{j\geq 1} e^{-j}j^{1-q}(j+1)^q\ll\frac{e^{tq(\gamma-2)}}{t^{3\gamma q/4}(2-\gamma q)}.
\]
For the remaining term, we use Hölder's inequality, Lemma \ref{lem:crudesupercritical} and Proposition \ref{prop:maximum} (noting that $Ae^{-2A}\leq e^{-A}$) to write
\[
    \ll \mathbb{P}\big(\exists h\notin G_{n_{K+1},\sigma_k}^{\text{max}}\big)^{(2-\gamma q)/2}  \mathbb{E}\big\{\mathcal{Z}_{\gamma,\sigma_k}^{2/\gamma}\big\}^{q{\gamma/2}} \ll e^{tq(\gamma-2)} e^{-\frac{(K+1)(2-\gamma q)}{2(2-\gamma q)}},
\]
and the claim follows by definition of $K$.
\end{proof}

\subsection{Proof of Proposition \ref{prop:maximum}}\label{sec:maximum} To begin with, note that by the proof of Lemma \ref{lem:goodevent} and the fact that $U_A(j)<U_A^{\max}(j)$ for $j\leq t/2$,
\begin{align*}
    &\,\,\sum_{j\leq t}\mathbb{P}\Big(\exists h:S_j(\sigma_k+ih)<L_A(j)\Big) \ll \sum_{A/4\leq j\leq t} e^{j-A^2/j-400j+40A} \ll e^{-2A-A^2/t},\\
    &\sum_{j\leq t/2}\mathbb{P}\Big(\exists h:S_j(\sigma_k+ih)>U_A^{\text{max}}(j)\Big) \ll e^{-2A-A^2/t},
\end{align*}
uniformly in $k$ and $t$.
It therefore suffices to show that with the same uniformity,
\begin{equation}\label{eq:3.1main}
        \sum_{t/2< j\leq  t}\mathbb{P}\Big(\max_{h\in[0,1]}S_\ell(\sigma_k+ih)\in [L_A(\ell), U_A^{\text{max}}(\ell)] \,\forall \ell\leq j-1, \max_{h\in [0,1]}S_{j}(\sigma_k+ih)>U_A^{\max}(j) \Big) \ll Ae^{-2A-A^2/t}.
\end{equation}
By translation invariance in law and a union bound, the left-hand side is bounded by
\[
    \sum_{t/2< j\leq  t}e^{j}\cdot \mathbb{P}\Big(S_\ell(\sigma_k)\in [L_A(\ell), U_A^{\text{max}}(\ell)] \,\forall \ell\leq j-1, \max_{h\in [0,e^{-j})}S_{j}(\sigma_k+ih)>U_A^{\max}(j) \Big).
\]
We then split the event in each summand into two: for $\max_{h\in [0,e^{-j})}S_j(\sigma_k+ih)$ to cross $U_A^{\text{max}}(j)$, either $S_j(\sigma_k)>U_A^{\text{max}}(j)$, or $S_j(\sigma_k)\leq U_A^{\text{max}}(j)$ while $\max_{h\in [0,e^{-j})}S_j(\sigma_k+ih)-S_j(\sigma_k
)>U_A^{\text{max}}(j)-S_j(\sigma_k)$.
This dichotomy was used by Arguin, Dubach and Hartung in \cite{ADH} to prove Proposition \ref{prop:maximum} for a Gaussian analogue of $S_j$.  In the first case, we have 
\begin{align}
    &\sum_{t/2<j\leq t} e^{j}\cdot\mathbb{P}\Big(S_\ell(\sigma_k)\in [L_A(\ell), U_A^{\text{max}}(\ell)] \,\forall \ell\leq j-1, S_{j}(\sigma_k)>U_A^{\max}(j) \Big)\label{eq:option1}\\
    &\leq \hspace{-20px}\sum_{\substack{L_A(j-1)\leq u\leq U^{\text{max}}_A(j-1)\\t/2<j\leq t}} \hspace{-20px}e^{j}\cdot\mathbb{P}\Big(S_\ell(\sigma_k)\in [L_A(\ell), U_A^{\text{max}}(\ell)] \,\forall \ell, S_{j-1}(\sigma_k)\in[u,u+1) \Big)\cdot\mathbb{P}\big(Y_j(\sigma_k)>U_A^{\max}(j)-u-1\big)\nonumber
\end{align}
where $Y_j=S_j-S_{j-1}$. On the one hand, a Chernoff bound and the estimate in Equation \eqref{eq:MGFinc} (summing only over $p\in (\exp(e^{j-1}),\exp(e^{j})]$) yield
\[
    \mathbb{P}\big(Y_j(\sigma_k)>U_A^{\max}(j)-u-1\big) \ll e^{-10(U_A^{\text{max}}(j)-u-1)} \ll e^{-10(U_A^{\max}(j-1)-u)}.
\]
On the other hand, Lemma \ref{lem:largedeviations} gives
\[
    \mathbb{P}\Big(S_\ell(\sigma_k)\in [L_A(\ell), U_A^{\text{max}}(\ell)] \,\forall \ell\leq j-1, S_{j-1}(\sigma_k)\in[u,u+1) \Big) \ll \frac{A\big(U_A^{\text{max}}(j-1)-u+C\big)}{j-1}\frac{e^{-u^2/(j-1)}}{\sqrt{j-1}}.
\]
(Note that this estimate holds for $u<0$ as well: in that case, we simply discard the first event and use Lemma \ref{lem:LDestimates} to get the bound $\ll {e^{-u^2/(j-1)}}/{\sqrt{j-1}}$, which is smaller than the right-hand side.) 
Using the change of variables $v=U_A^{\text{max}}(j-1)-u$, it follows that the sum in \eqref{eq:option1} is
\begin{equation}\label{eq:dichotomyIfinal}
    \ll A\sum_{t/2<j\leq t} \frac{e^{j}}{j^{3/2}}\sum_{v\geq 0} (v+C)e^{-10v-(U^{\text{max}}_A(j-1)-v)^2/(j-1)} \ll A\sum_{t/2<j\leq t} \frac{e^{j}}{j^{3/2}}e^{-U_A(j-1)^2/(j-1)}\sum_{v\geq 0}(v+C)e^{-5v}.
\end{equation}
Inserting the estimates $\sum_{v\geq 0}(v+C)e^{-5v}\leq 2C$ and
\[  
    -\frac{U_A(j-1)^2}{j-1}\leq-\frac{A^2}{t}-2A-(j-1)+\frac{3}{2}\log (j-1)-\mathcal{C}\cdot \log\big(1+(t-j+1)\big),
\]
we conclude that
\[
    A\sum_{t/2<j\leq t} \frac{e^{j}}{j^{3/2}}e^{-U_A(j-1)^2/(j-1)}\sum_{v\geq 0}(v+C)e^{-5v} \ll Ae^{-2A-A^2/t} \sum_{j\leq t} (t-j+1)^{-20} \ll Ae^{-2A-A^2/t},
\]
having picked $\mathcal{C}=10^3$.

What remains is to show that
\begin{align}
    &\sum_{t/2<j\leq t} e^{j}\cdot\mathbb{P}\Big(S_\ell(\sigma_k)\in [L_A(\ell), U_A^{\text{max}}(\ell)] \,\forall \ell\leq j, \max_{h\in [0,e^{-j})} S_{j}(\sigma_k+ih)>U_A^{\max}(j) \Big)\label{eq:option2}
\end{align}
satisfies the same bound. To this end, we partition according to $S_j(\sigma_k)\in [u,u+1)$ to get that the above is
\begin{align}\label{eq:Dichotomy2partitioned}
    \leq \sum_{t/2<j\leq t} \sum_{L_A(j)\leq u\leq U_A^{\max} (j)}e^{j}\cdot\mathbb{P}\Big(E_{j,u}^{(\sigma_k)},\max_{h\in [0,e^{-j})} S_{j}(\sigma_k+ih)-S_j(\sigma_k)>U_A^{\max}(j)-u-1 \Big),
\end{align}
where we used the shorthand
\begin{equation}\label{eq:E}
    E_{j,u}^{(\sigma_k)}:=\big\{S_\ell(\sigma_k)\in [L_A(\ell), U_A^{\text{max}}(\ell)] \,\forall \ell\leq j,\, S_j(\sigma_k)\in [u,u+1)\big\}.
\end{equation}
Our goal will be to show that each summand in \eqref{eq:Dichotomy2partitioned} satisfies the bound
\[
\ll e^{j-5(U_A^{\max}(j)-u)}\frac{A\big(U_A^{\text{max}}(j)-u+C\big)}{j}\frac{e^{-u^2/j}}{\sqrt{j}}.
\]
We will assume that $u<U_A^{\max}(j)-1$ without loss of generality, since the desired bound otherwise follows directly by applying Lemma \ref{lem:largedeviations} to $\mathbb{P}(E_{j,u}^{(\sigma_k)})$.

For the remaining $u$, we discretise the maximum over $h\in[0,e^{-j})$ using a same chaining argument similar to the proof of Lemma \ref{lem:chaining}. Following the argument therein from Equation \eqref{eq:qsum} to \eqref{eq:chainingfinal} (ignoring the sum over $q$ and the events $B_q$), we get the bound
\begin{align}
    \leq \sum_{m\geq 0} \sum_{\substack{v\in H_m}} e^j\cdot2\mathbb{P}\Big(E_{j,u}^{(\sigma_k)}, S_{j}(\sigma_k+iv)-S_j(\sigma_k+iv_*)>\frac{U_A^{\max}(j)-u-1}{e(m+1)^2} \Big),\label{eq:discretisedtriple}
\end{align}
where $v_*$ denotes the closest point to $v$ in $H_{m+1}$ which is not equal to $v$. (Should there be two such points, we let $v_*$ denote the smaller of the two.)
By a Chernoff bound, this is
\begin{equation}\label{eq:split}
    \leq \sum_{m\geq 0} \sum_{\substack{v\in H_m}}2\exp\!\Big({j-\lambda_v \frac{U_A^{\text{max}}(j)-u-1}{(e(m+1)^2)}}\Big)\mathbb{E}\Big\{e^{\lambda_v(S_{j}(\sigma_k+iv)-S_j(\sigma_k+iv_*))}\Big\}\cdot\mathbb{Q}_{v}\big(E_{j,u}^{(\sigma_k)}\big),
\end{equation}
where $\lambda_v=100(m+1)^{4}$ for $v\in H_m$, and $\mathbb{Q}_{v}$ is the tilted measure defined through
\[
    \frac{\mathrm{d}\mathbb{Q}_{v}}{\mathrm{d}\mathbb{P}}=\frac{\exp\!{\big({\lambda_v(S_{j}(\sigma_k+iv)-S_j(\sigma_k+iv_*))}}}{\mathbb{E}\Big\{\!\exp\!\big({\lambda_v(S_{j}(\sigma_k+iv)-S_j (\sigma_k+iv_*))}\big)\Big\}}.
\]
To bound the expectation in \eqref{eq:split}, note that $|v-v_*|^{-1}e^{-j}\asymp e^m$, and therefore that there exists a constant $c>0$ such that $\lambda_v\leq c|v-v_*|^{-1}e^{-j}$ for all $v\in \cup_m H_m$.
By the bound in \eqref{eq:O1MGF}, we thus get
\[
\mathbb{E}\!\big\{e^{\lambda_v (S_j(\sigma_k+iv)-S_j(\sigma_k+iv_*))}\big\} \ll 1
\]
uniformly in all parameters, and this can therefore be absorbed into the implied constant. Using the fact that $\# H_m\leq e^{m}+1$, it follows that \eqref{eq:split} is 
\begin{equation}\label{eq:3.8ext}
    \ll e^{j} \sup_{v\in \cup_{m}H_m}\mathbb{Q}_{v}\big(E_{j,u}^{(\sigma_k)}\big)\sum_{m\geq 0} e^{m-10(m+1)^2(U_A^{\text{max}}(j)-u-1)} \ll e^{j-5(U_A^{\max}(j)-u)} \sup_{v\in \cup_{m}H_m}\mathbb{Q}_{v}\big(E_{j,u}^{(\sigma_k)}\big).
\end{equation}
and we therefore need {uniform} estimates for the $\mathbb{Q}_{v}$-probability.

To do so, we make the crucial observation that for any such $v$, the  independence of $(f(p))_{p}$ persists under $\mathbb{Q}_{v}$. Recalling the definition of $E_{j,u}^{(\sigma_k)}$ in \eqref{eq:E}, we can therefore compare $\mathbb{Q}_{v}(E_{j,u}^{(\sigma_k)})$ to a Gaussian counterpart by proceeding as in the proof of Lemma \ref{lem:largedeviations} (up to \eqref{eq:Gaussiancomparisonfinal}), and leveraging the $\mathbb{Q}_{v}$-versions of Lemmas \ref{lem:LDestimates} and \ref{lem:berryesseen}. This yields
\begin{equation}\label{eq:drifted}
        \mathbb{Q}_{v}\big(E_{j,u}^{(\sigma_k)}\big)\ll  \mathbb{P}\Big(\mathcal{G}_j+\mu_j^{(v)}>u-2, \mathcal{G}_\ell+\mu_\ell^{(v)}\leq U_A^{\max}(\ell)+2, \forall \ell\leq j\Big),
\end{equation}
where $(\mathcal{G}_\ell)_{\ell}$ is the Gaussian random walk from Section \ref{sec:largedeviations}, and
\[
    \mu_\ell^{(v)}:=\sum_{s\leq \ell}\nu_s^{(v)},\quad \nu_s^{(v)}:= \mathbb{E}_{\mathbb{Q}_{v}}\big\{S_s(\sigma_k)-S_{s-1}(\sigma_k)\big\}.
\]
In other words, $\mathbb{Q}_{v}$ has the effect of adding a \textit{drift} to the random walk $\mathcal{G}_\ell$. However, since 
\[
   \sup_{k}\, \sup_{j\leq t} \sup_{v\in \cup_m H_m}|\mu_{j}^{(v)}|<D
\]
for some absolute constant $D>0$ by \eqref{eq:meanshift}, this will essentially have no effect on the final bound. Indeed, the right-hand side in \eqref{eq:drifted} is bounded by
\begin{align*}
  \mathbb{P}\Big(\mathcal{G}_j>u-2-\mu_j^{(v)}, \mathcal{G}_\ell\leq U_A^{\max}(\ell)-\mu_\ell^{(v)}, \forall \ell\leq j\Big) \leq \mathbb{P}\Big(\mathcal{G}_j>u-D', \mathcal{G}_j\leq U_A^{\max}(j)+D', \forall \ell\leq j\Big),
\end{align*}
for some $D'>0$, which we can bound by
\begin{align*}
        \ll \frac{A\big(U_A^{\text{max}}(j)-u+C'\big)}{j}\frac{e^{-u^2/j}}{\sqrt{j}}
\end{align*}
using Proposition \ref{cor:Hittingprob}, for some new constant $C'$ depending on $D'$. By \eqref{eq:3.8ext}, we conclude that \eqref{eq:Dichotomy2partitioned} is
 \[
    \ll \sum_{t/2<j\leq t} e^{j-5(U_A^{\max}(j)-u)}\frac{A\big(U_A^{\text{max}}(j)-u+C'\big)}{j}\frac{e^{-u^2/j}}{\sqrt{j}}.
 \]
which can be estimated as in \eqref{eq:dichotomyIfinal}.\qedhere \qed

\subsection{Corollaries} We end this section by showing how Corollaries \ref{cor:maximum} and \ref{cor:level} follow from the argument in the previous section. To get Corollary \ref{cor:maximum}, note that there exists a constant $C>0$ such that
\begin{equation}\label{eq:maxproof}
        \mathbb{P}\bigg(\max_{h\in [0,1] }|F_x(1/2+ih)|>\frac{\log x}{(\log\log x)^{3/4}}e^y\bigg)\leq \mathbb{P}\bigg(\max_{h\in [0,1]}S_{t}(1/2+ih)>t-\frac{3}{4}\log t+C+y\bigg)
\end{equation}
by \eqref{eq:trimming}, where $t=t(x)=\log\log x$. By Proposition \ref{prop:maximum}, this probability is 
\[
    \ll ye^{-2y-y^2/t}+ \mathbb{P}\bigg(\max_{h\in[0,1]}S_{t}(1/2+ih)>t-\frac{3}{4}\log t+C+y \cap \big\{[0,1]\subseteq  G_{y,1/2}^{\mathrm{max}}\big\}\bigg),
\]
which is $\ll ye^{-2y-y^2/t}$ by the same argument used to bound \eqref{eq:3.1main}.

For Corollary \ref{cor:level}, we let $E$ denote the event therein and bound the probability of its complement by
\[
    \mathbb{P}\big(\exists h\notin  G_{(C+\log A), 1/2}^{\max}\big)+\mathbb{P}\Big((\lnot E)\cap \big\{[0,1]\subseteq G_{(C+\log A), 1/2}^{\max}\big\}\Big),
\]
for the same constant $C$ as in \eqref{eq:maxproof}.
The first term is $\ll (\log A)/A$ by Proposition \ref{prop:maximum}. For the second, Markov's inequality and Fubini's theorem yield
\[
    \leq \frac{\mathbb{P}\Big(S_t(1/2)>t-\frac{3}{4}\log t+C+y, 0\in G_{(C+\log A),1/2}^{\max}\Big)}{(\log x)^{-1}A|\log A-y|e^{-2y-y^2/\log\log x}},
\]
which is $0$ when $y>\log A$, and $\ll (\log A)/A$ by Lemma \ref{lem:largedeviations} otherwise.

\section{Proof of Theorem \ref{thm:pseudomoments}}\label{sec:pseudomoments}

Building on the results in the previous section, we now prove Theorem \ref{thm:pseudomoments}. In this section, we let $\sigma_k:=1/2-2(k+1)/\log x$, and begin by using Proposition 6.6 in \cite{Gerspach}, by which
\[
    \Psi_{2q, \alpha}(x)\ll (\log x)^{q\alpha^2}+\frac{1}{(\log x)^q}\sum_{k\leq K} \mathbb{E}\bigg\{\bigg(\int_{0}^\infty\frac{|F_{x^{e^{-(k+1)}}}(\sigma_k+ih)|^{2\alpha}}{|2(k+1)/\log x+ih|^{2}}\mathrm{d}h\bigg)^q\bigg\}
\]
for $K=\lfloor \log\log\log x\rfloor$. Noting that $\alpha^2q <2q(\alpha-1)$ when $\alpha\in (1,2)$ and $q\in (0,2(\alpha-1)/\alpha^2)$, it suffices to show that the sum over $k$ satisfies the claimed bound.

Let $F^{(k)}:= F_{x^{e^{-(k+1)}}}$ for simplicity. By subadditivity of $x\mapsto x^q$, this sum is
\begin{align}\label{eq:first_pseudomoments}
    \leq \sum_{k\leq K}\frac{1}{(\log x)^q}\mathbb{E}\bigg\{\bigg(\int_{0}^1\frac{|F^{(k)}(\sigma_k+ih)|^{2\alpha}}{\big|\tfrac{2(k+1)}{\log x}+ih\big|^{2}}\mathrm{d}h\bigg)^q\bigg\} +  \sum_{k\leq K}\frac{1}{(\log x)^q}\mathbb{E}\bigg\{\bigg(\int_{1}^\infty{\frac{|F^{(k)}(\sigma_k+ih)|^{2\alpha}}{h^2}}\mathrm{d}h\bigg)^q\bigg\},
\end{align}
and Hölder's inequality and the translation invariance in law of $(F(\sigma_k+it), t\in [0,1])$ yield
\begin{align*}
    \mathbb{E}\bigg\{\bigg(\frac{1}{\log x}\int_{1}^\infty{\frac{|F^{(k)}(\sigma_k+ih)|^{2\alpha}}{h^2}}\mathrm{d}h\bigg)^{q}\bigg\} &\leq \mathbb{E}\bigg\{\bigg(\frac{1}{\log x}\int_{1}^\infty{\frac{|F^{(k)}(\sigma_k+ih)|^{2\alpha}}{h^2}}\mathrm{d}h\bigg)^{q^*}\bigg\}^{q/q^*}\\&\leq \bigg(\sum_{n\geq 1} \frac{1}{n^{2q^*}}\bigg)\mathbb{E}\bigg\{\bigg(\frac{1}{\log x}\int_{0}^1{{|F^{(k)}(\sigma_k+ih)|^{2\alpha}}}\mathrm{d}h\bigg)^{q^*}\bigg\}^{q/q^*}
\end{align*}
for any $q^*\in (1/2,1/\alpha)$ and $k\leq K$. The sum over $n$ being finite, we conclude using Theorem \ref{thm:mainbound} that\footnote{Note that the shift from $1/2$ here is by $2(k+1)/\log x$ rather than $(k+1)/\log x$, but one straightforwardly checks that the proof of Theorem \ref{thm:mainbound} remains valid with this choice.}
\[
    \sum_{k\leq K}\frac{1}{(\log x)^q}\mathbb{E}\bigg\{\bigg(\int_{1}^\infty{\frac{|F^{(k)}(\sigma_k+ih)|^{2\alpha}}{h^2}}\mathrm{d}h\bigg)^q\bigg\}\ll \sum_{k\leq K}\bigg(\frac{e^{-kq^*}(\log x)^{2(\alpha-1)q^*}}{(\log \log x)^{3q^*\alpha/2}}\bigg)^{q/q^*} = \frac{(\log x)^{2(\alpha-1)q}}{(\log \log x)^{3q\alpha/2}}.
\]

To handle the first sum in \eqref{eq:first_pseudomoments}, we decompose the range of integration into $e$-adic intervals and once again use subadditivity of $x\mapsto x^q$. This yields
\begin{align*}
    \sum_{k\leq K}\frac{1}{(\log x)^q}\mathbb{E}\bigg\{\bigg(\int_{0}^1\frac{|F^{(k)}(\sigma_k+ih)|^{2\alpha}}{|\tfrac{2(k+1)}{\log x}+ih|^{2}}\mathrm{d}h\bigg)^q\bigg\} \ll \sum_{k\leq K} \sum_{j\leq J} \frac{T_{j-1}^{-2q}}{(\log x)^q} \mathbb{E}\bigg\{\bigg(\int_{T_{j-1}}^{T_{j}}|F^{(k)}(\sigma_k+ih)|^{2\alpha}\mathrm{d}h\bigg)^q\bigg\}
\end{align*}
where $T_j=(e^{2(k+1)}/\log x)e^{j}$ for each $j\geq 1$, and $T_{-1}:=0$. The sum is taken up to $J_k=J(k,x)$, defined as the smallest integer for which $e^{J+2(k+1)}/\log x\geq 1$. 

To bound each summand, we make the observation that on $[T_{j-1}, T_j]$, the contribution to $|F^{(k)}(\sigma_k+ih)|^{2\alpha}$ coming from primes up to $x_j=e^{1/T_j}$ is roughly constant over the range of integration; it should approximately equal $|F_{x_j}(\sigma_j)|^{2\alpha q}$, as suggested by Lemma \ref{lem:chaining}. To leverage this fact, we introduce the a family of tilted measures $(\mathbb{P}_j)_{j\leq J_k}$, defined through 
\[
    \frac{\mathrm{d}\mathbb{P}_j}{\mathrm{d}\mathbb{P}}=\frac{|F_{x_j}(\sigma_k)|^{2\alpha q}}{\mathbb{E}\{|F_{x_j}(\sigma_k)|^{2\alpha q}\}}.
\]
 For each $j\leq J_k$, 
\begin{align*}
    \mathbb{E}\bigg\{\bigg(\int_{T_{j-1}}^{T_{j}}|F^{(k)}(\sigma_k+ih)|^{2\alpha}\mathrm{d}h\bigg)^q\bigg\} = \mathbb{E}\Big\{\big|F_{x_j}(\sigma_k)\big|^{2\alpha q}\Big\}\cdot \mathbb{E}_{\mathbb{P}_j}\bigg\{\bigg(\int_{T_{j-1}}^{T_{j}}\frac{|F^{(k)}(\sigma_k+ih)|^{2\alpha}}{|F_{x_j}(\sigma_k+ih)|^{2\alpha}}\mathcal{E}_T(h)^{2\alpha}\mathrm{d}h\bigg)^q\bigg\}
\end{align*}
where $\mathcal{E}_T(h):=|F_{x_j}(\sigma_k+ih)|/|F_{x_j}(\sigma_k)|$ can be seen as an error term. We can then condition on the $\sigma$-algebra generated by $(f(p))_{x_j<p\leq x^{e^{-(k+1)}}}$ and use Jensen's inequality to get the bound
\[
    \ll \mathbb{E}\Big\{\big|F_{x_j}(\sigma_k)\big|^{2\alpha q}\Big\}\cdot \mathbb{E}\bigg\{\bigg(\int_{T_{j-1}}^{T_{j}}\frac{|F^{(k)}(\sigma_k+ih)|^{2\alpha}}{|F_{x_j}(\sigma_k+ih)|^{2\alpha}}\mathbb{E}_{\mathbb{P}_{j}}\big\{\mathcal{E}_T(h)^{2\alpha}\big\}\mathrm{d}h\bigg)^q\bigg\},
\]
noting that $(|F^{(k)}(\sigma_k+ih)|/|F_{x_j}(\sigma_k+ih)|)_{h\in [T_{j-1},T_j]}$ is measurable with respect to said $\sigma$-algebra, and that its law under $\mathbb{P}_j$ is the same as under $\mathbb{P}$ by said independence. By the moment bound in \eqref{eq:ERRORBOUND}, 
\[
    \sup_{j\leq J_k}\sup_{h<T_j} \mathbb{E}_{\mathbb{P}_j}\big\{\mathcal{E}_T(h)^{2\alpha}\big\}\ll 1,
\]
Since $\mathbb{E}\{|F_{x_j}(\sigma_k)|^{2\alpha q}\}\ll T_j^{-(\alpha q)^2}$ uniformly in $j$ and $k$ by Lemma \ref{lem:LDestimates}, it follows that
\begin{align*}
    \sum_{\substack{k\leq K\\ j\leq J_k}} \frac{T_{j-1}^{-2q}}{(\log x)^q} \mathbb{E}\bigg\{\bigg(\int_{T_{j-1}}^{T_{j}}|F^{(k)}(\sigma_k+ih)|^{2\alpha}\mathrm{d}h\bigg)^q\bigg\} \ll  \sum_{\substack{k\leq K\\ j\leq J_k}} \frac{T_{j}^{-2q-(\alpha q)^2}}{(\log x)^q}\mathbb{E}\bigg\{\bigg(\int_{T_{j-1}}^{T_{j}}\frac{|F^{(k)}(\sigma_k+ih)|^{2\alpha}}{|F_{x_j}(\sigma_k+ih)|^{2\alpha}}\mathrm{d}h\bigg)^q\bigg\} 
\end{align*}
We're left with the task of bounding
\begin{align}\label{eq:shiftedGMC}
        \mathbb{E}\bigg\{\bigg(\int_{T_{j-1}}^{T_{j}}\frac{|F^{(k)}(\sigma_k+ih)|^{2\alpha}}{|F_{x_j}(\sigma_k+ih)|^{2\alpha}}\mathrm{d}h\bigg)^q\bigg\} = \mathbb{E}\bigg\{\bigg(\int_{0}^{(1-1/e)T_{j}}\frac{|F^{(k)}(\sigma_k+ih)|^{2\alpha}}{|F_{x_j}(\sigma_k+ih)|^{2\alpha}}\mathrm{d}h\bigg)^q\bigg\}.
\end{align}
To that end, we define the \textit{shifted} field
\[
        {S}_{k}^{(j)}(\sigma_k+ih)=S_{k+t_j}\big(\sigma_k+ie^{t_j}h\big)-S_{t_j}\big(\sigma_k+ie^{t_j}h\big),\quad t_j:=\log\log x_j,
\]
where $S_k(s)$ is defined as in \eqref{eq:walk}. By discarding higher order terms in the expansion of $\log |F^{(k)}/F_{x_j}|$ (cf.\! Equation \eqref{eq:trimming} and using the  change of variables $h\mapsto h/((1-1/e)T_j)$, the expectation in \eqref{eq:shiftedGMC} is  
\[
    \leq T_{j}^q\cdot \mathbb{E}\bigg\{\bigg(\int_{0}^{1}e^{2\alpha S_{t-t_j}^{(j)}(\sigma_k+ih)}\mathrm{d}h\bigg)^q\bigg\}, \quad \text{ where } t=\log\log x-(k+1).
\]
 This can then be studied exactly as in the proof of Theorem \ref{thm:mainbound}, replacing every occurence of $S$ with $S^{(j)}$. Indeed, $S_{t-t_j}^{(j)}$ is now a sum of $t-t_j$ independent increments with variances
\[
    \sum_{\exp(e^{t_j+\ell-1})< p \leq \exp(e^{t_j+\ell})}\frac{1}{2p^{2\sigma}}+\frac{1}{8p^{4\sigma}},\quad \ell\leq t-t_j,
\]
and the $t_j$-shift in the range of $\log\log p$ in this sum only improves the error terms in the estimates in the appendix. We conclude that
\[
    T_j^q\cdot \mathbb{E}\bigg\{\bigg(\int_{0}^{1}e^{2\alpha S_{t-t_j}^{(j)}(\sigma_k+ih)}\mathrm{d}h\bigg)^q\bigg\} \ll  T_j^{q}(e^{jq})\cdot\frac{e^{j\cdot2q(\alpha-1)}}{j^{q(3\alpha/2)}}, 
\]
{uniformly in $j$ and $k$}, and in turn that 
\begin{align*}
    \Psi_{2\alpha, q}(x)&\ll (\log x)^{q\alpha^2}+\sum_{k\leq K} \sum_{j\leq J_k} e^{-kq}T_{j}^{-(\alpha q)^2}  \frac{e^{j\cdot 2q(\alpha-1)}}{j^{q(3\alpha/2)}}\ll (\log x)^{q\alpha^2}+\sum_{k} e^{-Ak}\frac{e^{J_k\cdot 2q(\alpha-1)}}{J_k^{3q\alpha/2}}
\end{align*}
where $A>0$ is a constant. The theorem then follows by recalling that $J_k=\lceil \log\log x-2(k+1)\rceil$.
\appendix

\section{Gaussian comparison}
Recall that $(Z_p)_p$ is a collection of {i.i.d.} Steinhaus random variables. It will be convenient to introduce
\[
    X_p^{(\sigma)}(h):=\Re\Big(\frac{Z_p}{p^{\sigma+ih}}+\frac{Z_p^2}{p^{2(\sigma+ih)}}\Big), \quad I_j:=(\exp(e^{j-1}),\exp(e^j)],
\]
so that $S_j(\sigma+ih):=\sum_{p\in I_j}X_p^{(\sigma)}(h).$ We also let $V_j(\sigma):= \sum_{p\in I_j} \frac{1}{2p^{2\sigma}}+\frac{1}{8p^{4\sigma}}.$
\begin{lemma}[Prime number theorem estimates]\label{lem:PNT}
Uniformly in $b>a>2$, there exists a $c>0$ for which
    \begin{equation}
        \sum_{a<p\leq b} \frac{1}{p}= \log\log b-\log\log a+O\big(e^{-c\sqrt{\log a}}\big),
    \end{equation}
and uniformly in $\sigma\in (0,1/2)$, 
\begin{equation}
    \sum_{a<p\leq b} \frac{1}{p^{2\sigma}} = \mathrm{Ei}\big((1-2\sigma)\log b\big)-\mathrm{Ei}\big((1-2\sigma)\log a\big)+O\big(e^{-c\sqrt{\log a}}\big)
\end{equation}
where $\mathrm{Ei}(x):=\int_{-\infty}^x(e^{s}/s)\mathrm{d}s$.
We also have that $\sum_{a\leq p\leq b}\frac{(\log p)^m}{p}=O\big((\log b)^m\big)$ for any $1\leq a\leq b$.
\end{lemma}
\begin{proof}
    This follows straightforwardly from Theorem 6.9 in \cite{MontgomeryVaughan} using integration by parts.
\end{proof}

\begin{lemma}[Large deviation estimates]\label{lem:LDestimates} Let $C>0$ be an arbitrary constant. Then uniformly in all large $x$, $0\leq k\leq \log\log\log x$, $\sigma=1/2-k/\log x$, $1\leq j\leq \log\log x-k+1$, $|\alpha|\leq C$, and ${|v|\leq Cj}$,
\begin{equation}\label{eq:density}
        \mathbb{E}\big\{e^{\alpha S_j(\sigma)}\big\}\ll e^{\frac{\alpha^2}{4}j},\quad \mathrm{and} \quad \mathbb{P}\big(S_{j}(\sigma)\in [v,v+1)\big) \ll \frac{e^{-v^2/j}}{\sqrt{j}}.
\end{equation}
Furthermore, the same bounds hold under the measures $\mathbb{Q}=\mathbb{Q}_{h_1,h_2,\sigma,\beta}$ defined through
\[
    \frac{\mathrm{d}\mathbb{Q}}{\mathrm{d}\mathbb{P}}=\frac{\exp\big(\beta(S_j(\sigma+ih_1)-S_j(\sigma+ih_2))\big)}{\mathbb{E}\big\{\!\exp\big(\beta(S_j(\sigma+ih_1)-S_j(\sigma+ih_2))\big)\!\big\}}
\]
uniformly in $0\leq \beta\leq C|h_1-h_2|^{-1}e^{-j}$ and $h_1,h_2\in [-e^{-j-1},e^{-j-1}]$.
\end{lemma}
\begin{proof}
Since $|\alpha X_p^{(\sigma)}(0)|$ is bounded uniformly in $p\geq 2$ by our assumption on $\alpha$,  we can write
\[
    \log \mathbb{E}\big\{e^{\alpha S_j(\sigma)}\big\}=\sum_{p\leq \exp(e^j)} \log \mathbb{E}\big\{e^{\alpha X_p^{(\sigma)}(0)}\big\}=\sum_{p\leq \exp(e^j)} \log \Big(1+\frac{\alpha^2}{4p^{2\sigma}}+O\big(\alpha^3 p^{-3\sigma}\big)\Big)
\]
by Taylor expanding the exponential. Using Lemma \ref{lem:PNT} and the expansions $\log(1+x)=x+O(x^2)$ and $\mathrm{Ei}(x)=\gamma+\ln|x|+x+O(x^2)$,  
\begin{equation}\label{eq:MGFinc}
     \log \mathbb{E}\big\{e^{\alpha S_j(\sigma)}\big\}=\frac{\alpha^2}{4}\sum_{p\leq \exp(e^j)}\frac{1}{p^{2\sigma}}+O(\alpha^6p^{-3\sigma})\leq \frac{\alpha^2}{4}j+C',
\end{equation}
for some constant $C'$ (depending on $C$). It follows that $\mathbb{E}\{e^{\alpha S_j(\sigma)}\}\ll e^{(\alpha^2/4)j}$ for any $|\alpha|\leq C$. 

To estimate the probability in \eqref{eq:density}, we rewrite it as
\begin{equation}\label{eq:LD1}
    \mathbb{E}\big\{e^{\alpha S_j(\sigma)}\big\} \mathbb{E}_{\tilde{\mathbb{\mathbb{P}}}}\big\{e^{-\alpha S_j(\sigma)}\mathbf{1}\big(S_{j}(\sigma)\in [v,v+1)\big)\big\} \leq \mathbb{E}\big\{e^{\alpha S_j(\sigma)}e^{-\alpha v+\mathbf{1}(v<0)}\big\}\tilde{\mathbb{\mathbb{P}}}\big(S_{j}(\sigma)\in [v,v+1)\big)\big\}
\end{equation}
for $\alpha=2v/j$, where $\tilde{\mathbb{\mathbb{P}}}$ is the measure given by $\mathrm{d}\tilde{\mathbb{\mathbb{P}}}/{\mathrm{d}\mathbb{P}}={e^{\alpha S_{j}(\sigma)}}/{\mathbb{E}\{e^{\alpha S_{j}(\sigma)}}\}$. The expectation is $\ll e^{-v^2/j}$ by the earlier estimate $\mathbb{E}\{e^{\alpha S_j(\sigma)}\}\ll e^{(\alpha^2/4)j}$. To estimate the remaining probability, note that deterministically,
\[  
    \mathbb{E}_{\tilde{\mathbb{P}}}\Big\{\Big(S_j(\sigma)-\mathbb{E}_{\tilde{\mathbb{P}}}\big\{S_j(\sigma)\big\}\Big)^3\Big\}\leq C'\sum_{p} p^{-3\sigma} <\infty
\]
for some constant $C'>0$. Furthermore, there exists a constant $C_0$ depending only on $C$ such that 
\[  
    \frac{e^{\alpha X_p^{(\sigma)}(0)}}{\mathbb{E}\big\{e^{\alpha X_p^{(\sigma)}(0)}\big\}}=1+\alpha X_p^{(\sigma)}(0)+\frac{\alpha^2}{2}\Big(X_p^{(\sigma)}(0)^2-\mathbb{E}\big\{X_p^{(\sigma)}(0)^2\big\}\Big)+O(\alpha^3p^{-3\sigma}\big)
\]
uniformly for $p>C_0$, and we may assume that $C_0<\exp(e^j)$ without loss of generality (the desired bound otherwise follows by bounding the probability in \eqref{eq:LD1} by $1$). It follows that
\begin{align*}
    \mathrm{Var}_{\tilde{\mathbb{P}}}\big(S_j(\sigma)\big) &= O_C(1)+ \sum_{C_0<p\leq \exp(e^j)} \mathrm{Var}_{\tilde{\mathbb{P}}}\big(X_p^{(\sigma)}(0)\big) \\
    &= O_C(1)+\sum_{C_0<p\leq \exp(e^j)}\mathbb{E}\bigg\{\frac{e^{\alpha X_p^{(\sigma)}(0)}}{\mathbb{E}\big\{e^{\alpha X_p^{(\sigma)}(0)}\big\}}\Big(X_p^{(\sigma)}(0)-\mathbb{E}\big\{X_p^{(\sigma)}(0)\big\}\Big)^2\bigg\}\\
    &= O_C\big(1+\sum_p p^{-4\sigma}\big)+\mathrm{Var}_{\mathbb{P}}\big(S_j(0)\big) \asymp \mathrm{Var}_{\mathbb{P}}\big(S_j(0)\big).
\end{align*}
Using a standard Berry-Esseen bound and Lemma \ref{lem:PNT}, we conclude that
\[
    \tilde{\mathbb{\mathbb{P}}}\big(S_{j}(\sigma)\in [v,v+\Delta^{-1})\big)\ll \mathrm{Var}_{{\mathbb{P}}}\big(S_j(\sigma)\big)^{-1/2}\ll j^{-1/2}.
\]
The claim for $\mathbb{Q}$ follows from the same argument, provided one is armed with mean and variance estimates for $S_j(\sigma)$ under $\mathbb{Q}$, as well as a variance estimate under $\tilde{\mathbb{Q}}$ where $\mathrm{d}\tilde{\mathbb{\mathbb{Q}}}/{\mathrm{d}\mathbb{Q}}={e^{\alpha S_{j}(\sigma)}}/{\mathbb{E}_{\mathbb{Q}}\big\{e^{\alpha S_{j}(\sigma)}}\big\}$. We compute these directly, using the expansion
\[
    \frac{e^{\beta (X_p^{(\sigma)}(h_1)-X_p^{(\sigma)}(h_2))}}{\mathbb{E}\big\{e^{\beta (X_p^{(\sigma)}(h_1)-X_p^{(\sigma)}(h_2))}\big\}}=1+\beta \big(X_p^{(\sigma)}(h_1)-X_p^{(\sigma)}(h_2)\big)+O\Big(\beta^2 \big(X_p^{(\sigma)}(h_1)-X_p^{(\sigma)}(h_2)\big)^2\Big),
\]
for $p\leq \exp(e^j)$ large enough, and noting that the error term is
\[
   \ll  \frac{\beta^2|p^{-ih_1}-p^{-ih_2}|^2}{p^{2\sigma}} \leq \frac{\beta^2|h_1-h_2|^2(\log p)^2}{p^{2\sigma}} \ll \frac{1}{p^{2\sigma}}
\]
uniformly for $p\leq \exp(e^j)$ by our assumptions on $\beta$ and $|h_1-h_2|$. We therefore have

\begin{align}\label{eq:meanshift}
    \mathbb{E}_{\mathbb{Q}}\big\{S_j(\sigma)\big\} 
    &= O\Big(\sum_{p\leq \exp(e^j)} p^{-3\sigma}\Big)+\beta\sum_{p\leq \exp(e^j)} \mathbb{E}\big\{X_p^{(\sigma)}(0)\big(X_p^{(\sigma)}(h_1)-X_p^{(\sigma)}(h_2)\big)\big\}\notag\\
    &=O(1)+\frac{\beta}{2}\sum_{p\leq \exp(e^j)}
    \frac{\cos(h_1\log p)-\cos(h_2\log p)}{p^{2\sigma}}  \ll |h_1-h_2|\sum_{p\leq \exp(e^j)} \frac{\log p}{p} +1,
\end{align}
which is $O(1)$ by Lemma \ref{lem:PNT} since $|h_1-h_2|\leq e^{-j}$. Similarly, we find that
\begin{align}
    \mathrm{Var}_{\mathbb{Q}}\big(S_j(\sigma)\big)
    &= \mathrm{Var}_{\mathbb{P}}\big(S_j(\sigma)\big)
    + O\Big(\beta\,|h_1-h_2|\sum_{p} p^{-3\sigma}\log p\Big)
    + O\Big(\beta^2\sum_{p} p^{-4\sigma}\Big) \notag\\
    &= \mathrm{Var}_{\mathbb{P}}\big(S_j(\sigma)\big) + O(1),\label{eq:Qvar}
\end{align}
and $\mathrm{Var}_{\tilde{\mathbb{Q}}}\big(S_j(\sigma)\big)\asymp \mathrm{Var}_{\mathbb{Q}}\big(S_j(\sigma)\big)$.
\qedhere
\end{proof}
\begin{remark}
Letting $\sigma=1/2-k/\log x$ be as in Lemma \ref{lem:LDestimates}, the same argument also yields the bound
\begin{equation}\label{eq:ERRORBOUND}
    \sup_{k\leq \lfloor\log\log\log x\rfloor}\sup_{j\leq \log\log x-k}\sup_{0<h<e^{-j}}\frac{\mathbb{E}\big\{e^{2\alpha S_j(\sigma+ih)-2\alpha(1-q)S_j(\sigma)}\big\}}{\mathbb{E}\big\{e^{2\alpha q S_j(\sigma)}\big\}} <\infty
\end{equation}
for $q<1$ and $\alpha<2$ fixed. This estimate is needed in the proof of Theorem \ref{thm:pseudomoments}.
\end{remark}
\begin{lemma}[Berry-Esseen estimate]\label{lem:berryesseen}  Let $x>0$ be large enough and $h\in [0,1]$ be arbitrary. Let $\sigma=1/2-k/\log x$ for $0\leq k\leq \log\log\log x$. Then there exists a constant $c>0$ such that for $1\leq j\leq \log\log x-k$ and any interval $A\subseteq \mathbb{R}$,
\[
    \mathbb{P}\Big(\big(S_j(\sigma+ih)-S_{j-1}(\sigma+ih)\big)\in A\Big)=\mathbb{P}\big(\mathcal{N}_j\in A\big)+O\big(e^{-ce^{j/2}}\big),
\]
where $\mathcal{N}_j$ is a real, centered Gaussian random variable with variance $V_j(\sigma)$.

Furthermore, for any $\mathbb{Q}$ defined as in the statement of Lemma \ref{lem:LDestimates}, the same estimate holds under $\mathbb{Q}$ upon replacing $\mathcal{N}_j$ on the right-hand side by $\mathcal{N}_j+\nu_j$, where $\nu_j:=\mathbb{E}_{\mathbb{Q}}\{S_j(\sigma+ih)-S_{j-1}(\sigma+ih)\}$
\end{lemma}
\begin{proof}
    We begin with the claim for $\mathbb{P}$, which is essentially Lemma 20 in \cite{FHK1}. Note that the $X_p^{(\sigma)}(h)$ involved are centered, have variance in $[C^{-1},C]$ for some $C>0$, and satisfy $|X_p^{(\sigma)}(h)|<Cp^{-\sigma}$ deterministically. The Berry-Esseen theorem (Corollary 17.2 in \cite{Bhattacharya}) thus yields
    \begin{align*}
         \Big|\mathbb{P}\Big(\big(S_j(\sigma+ih)-S_{j-1}(\sigma+ih)\big)\in A\Big)-\mathbb{P}\big(\mathcal{N}_j\in A\big)\Big|&\ll \sum_{p\in I_j}p^{-3\sigma}= O\big(e^{-ce^{j/2}}\big).
    \end{align*}
    Under $\mathbb{Q}$, we simply apply the Berry-Esseen theorem to $S_j(\sigma+ih)-S_{j-1}(\sigma+ih)-\nu_j=\sum_{p\in I_j} X_p^{(\sigma)}(h)-\mathbb{E}_{\mathbb{Q}}\big\{X_p^{(\sigma)}(h)\big\}$.
\end{proof}

\section{Discretisation}

\begin{lemma}[Two-point estimates]\label{lem:LD} Let $C>0$ be arbitrary and $x>0$ be large enough. Let $0\leq k\leq \log\log\log x$, $\sigma=1/2-k/\log x$, $j\leq \log\log x-k$. Finally, let $-e^{-j-1}\leq h_1,h_2\leq e^{-j-1}$, $0\leq V_1\leq Cj$, and $0\leq V_2\leq e^{2j}$. Then uniformly in all of these parameter ranges,
\[
     \mathbb{P}\big(S_j(\sigma)\geq V_1, S_{j}(\sigma+ih_1)-S_{j}(\sigma+ih_2)\geq V_2\big)\ll_C  \exp\Big(\!-\frac{V_1^2}{j}-\frac{cV_2^{3/2}}{e^{j}|h_2-h_1|}\Big)
\]
for a constant $c>0$ which depends on $C$. Furthermore, if $\lambda \leq C|h_1-h_2|^{-1}e^{-j}$,
\begin{equation}\label{eq:O1MGF}
        \mathbb{E}\big\{\!\exp\!\big(\lambda(S_j(\sigma+ih_1)-S_j(\sigma+ih_2)\big)\big\}\ll 1.
\end{equation}
\end{lemma}
\begin{proof}
    By a Chernoff bound, for any choice of $\lambda_1,\lambda_2>0$, this probability is bounded by
    \begin{equation}\label{eq:chernoff}
            \mathbb{E}\Big\{\!\exp\Big({\lambda_1 S_k(\sigma)+\lambda_2\big(S_k(\sigma+ih_1)-S_k(\sigma+ih_2)\big)}\Big)\Big\}\exp\big({-\lambda_1V_1-\lambda_2V_2}\big).
    \end{equation}
    If $\lambda_1\leq C$ and $1\leq \lambda_2\leq |h_2-h_1|^{-1}$, we can estimate this Laplace transform by Taylor expanding the exponential as in the proof of Lemma \ref{lem:largedeviations}. This yields 
    \begin{equation*}
                             \mathbb{E}\Big\{\!\exp\Big({\lambda_1 S_j(\sigma)+\lambda_2\big(S_j(\sigma+ih_1)-S_j(\sigma+ih_2)\big)}\Big)\Big\}\ll_C  \exp\bigg(\frac{\lambda_1^2}{4}j+c\lambda_2e^{j}|h_1-h_2|+c\big(\lambda_2e^j |h_2-h_1|\big)^2\bigg)
    \end{equation*}
    for some constant $c>0$ depending on $C$. The first claim follows by using this in \eqref{eq:chernoff} and picking
    \[
        \lambda_1= V_1/j, \quad \lambda_2= c\sqrt{V_2}\cdot e^{-j}|h_2-h_1|^{-1},
    \]
    assuming that $V_2$ is greater than a sufficiently large constant times $e^j|h_2-h_1|$ (and in turn $e^{2j}|h_1-h_2|^2$). We can do this without loss of generality, since the desired bound reduces to \eqref{eq:density} for smaller $V_2$. The second claim follows by a similar argument.
\end{proof}

\begin{lemma}[Maximum bound]\label{lem:chaining} Let $C>0$ be arbitrary. Then uniformly in all large $x$, $0\leq k\leq \log\log\log x$, $\sigma=1/2-k/\log x$, $j\leq \log\log x-k$ and $0\leq V\leq Cj$,
    \begin{equation}\label{eq:maximumbound}
                \mathbb{P}\Big(\max_{h\in [0,e^{-j})}S_j(\sigma+ih)>V\Big)\ll_C \exp(-V^2/j).
    \end{equation}
\end{lemma}
\begin{proof} 
 By the large deviations estimate in \eqref{eq:density}, 
    \[
        \mathbb{P}\Big(\max_{h\in [0,e^{-j})}S_j(\sigma+ih)>V\Big)\ll_C \exp(-V^2/j)+\mathbb{P}\Big(\max_{h\in [0,e^{-j})}S_j(\sigma+ih)>V, S_j(\sigma)\leq V-2\Big).
    \]
    To bound the probability on the right-hand side, we use a standard chaining argument which was  adapted to this setting in \cite{ABH} (Proposition 2.5). We include this argument here for completeness.
    Let $$H_\ell=[0,e^{-j}]\cap e^{-j-\ell}\mathbb{Z}, \quad\ell\geq 0,$$ and assume without loss of generality that $V$ is an integer. For every $q\in \{0,1,...,V-3\}$, let $B_q$ be the event that $S_j(\sigma)\in [V-q-1,V-q]$, and $B_{V-2}$ be the event that $S_j(\sigma)\leq 0$. Then
    \begin{align}\label{eq:qsum}
        \mathbb{P}\Big(\max_{h\in [0,e^{-j})}S_j(\sigma+ih)>V, S_j(\sigma)\leq V-2\Big) \leq \sum_{q=0}^{V-2} \mathbb{P}\Big(B_q\cap \Big\{\max_{h\in [0,e^{-j}]}S_j(\sigma+ih)-S_j(\sigma)\geq q\Big\}\Big).
    \end{align}
    We now decompose the summands on the right-hand side. Let $(h_\ell)_{\ell\geq 0}$ be an $e$-adic sequence tending to $h$ with $\ell$, satisfying $h_0=0$ and $h_\ell\in H_{\ell}$. Then by continuity of $h\mapsto S_j(\sigma+ih)$,
    \[
        S_j(\sigma+ih)-S_j(\sigma)=\sum_{\ell\geq 0} S_j(\sigma+ih_{\ell+1})-S_j(\sigma+ih_\ell)
    \]
    and the series on the right-hand side converges almost surely.  Furthermore, since $\sum_{\ell=0}^\infty\frac{1}{e(\ell+1)^2}\leq 1$, 
    \[
        \big\{S_j(\sigma+ih)-S_j(\sigma)\geq q\big\} \subset \bigcup_{\ell\geq 0}\Big\{S_j(\sigma+ih_{\ell+1})-S_j(\sigma+ih_\ell)\geq \frac{q}{e(\ell+1)^2}\Big\}.
    \]
    Since this holds for any such $(h_\ell)_{\ell\geq 0}$, we conclude by a union bound that the right-hand side in \eqref{eq:qsum} is
    \begin{equation}\label{eq:chainingfinal}
        \leq \sum_{q=0}^{V-2}\sum_{\ell\geq 0}\sum_{h\in H_\ell}2\cdot \mathbb{P}\Big(B_q\cap\Big\{ S_j(\sigma+ih)-S_j(
        \sigma+ih_*)\geq \frac{q}{e(\ell+1)^2}\Big\}\Big),
    \end{equation}
    where $h_*\neq h$ denotes the closest point to $h$ in $H_{\ell+1}$. (Should there be two such points, we let $h_*$ be the smaller of the two and note that the bound still holds due to the additional factor of $2$, since the probability only depends on $|h-h_*|$.)
    Noting that $\#H_{\ell}=e^\ell+1$, we can use the joint large deviations estimate of Lemma \ref{lem:LD} to estimate each summand, and conclude that \eqref{eq:chainingfinal} is
    \[
        \ll {\sum_{q=0}^{V-2}\sum_{\ell\geq 0} e^\ell \exp\Big(-\frac{(V-q-1)^2}{j}-ce^\ell\frac{q^{3/2}}{(\ell+1)^3}\Big) \ll \sum_{q=0}^{V-2}  e^{-{(V-q-1)^2}/{j}-cq^{3/2}}\ll e^{-V^2/j}.}\qedhere
    \]
\end{proof}
\begin{remark}\label{rem:symmetry}
    The conclusions of both Lemma \ref{lem:LD} and Lemma \ref{lem:chaining} hold upon replacing $S_j$ by $-S_j$, by the same proof. 
\end{remark}

\section{Ballot theorem}
\begin{proposition}\label{cor:Hittingprob}
Fix $0<\kappa< 1$ and $\delta>0$. Then there exist constants $C,C'>0$ depending only on $\kappa$ and $\delta$ such that the following holds. Let $\{\mathcal{N}_j\}_{j\geq 1}$ be a collection of independent, centered Gaussian random variables with variances $\mathbb{E}\{\mathcal{N}_j^2\}\in(\kappa,\kappa^{-1})$ for each $j$, { satisfying  $|\sum_{j\leq k} \mathbb{E}\{\mathcal{N}_j^2\}-k/2|<\delta$ for all $k\geq 1$. Let $\mathcal{G}_k=\sum_{j\leq k}\mathcal{N}_j$ for each $k\geq 1$. For any $t>1$, $4<A\leq t$, let $U_A(s)=\infty$ for $s< A/4$, and for $s\geq A/4$, let $U_A$ be one of the following two functions:}
\[
    U_A(s)=A+s+2\log\big(1+s\land(t-s)\big)  \text{ or } U_A(s)=A+s\Big(1-\frac{3}{4}\frac{\log t}{t}\Big)+10^3\log\big(1+s\land(t-s)\big).
\]
Then, for either choice of $U_A$, and for any $t/2\leq k\leq t$ and $w\in [0, U_A(k))$,
    \[
        \mathbb{P}\Big(\mathcal{G}_k>w\text{ $\mathrm{and}$ }\,\mathcal{G}_j\leq U_A(j)+1 \mathrm{\,\, for\,\, all\,\, } j\leq k \Big)\leq C\cdot \frac{A\big(U_A(k)-w+C'\big)}{k}\frac{e^{-w^2/k}}{\sqrt{k}}
    \]
   
\end{proposition}
\begin{proof}
   This follows directly from Proposition 5 in \cite{FHK1} by conditioning on the values of the random walk at times $\lceil A/4\rceil$ and $k$. 
\end{proof}

\bibliographystyle{abbrv}
\bibliography{biblio}

\end{document}